\documentclass[11pt,a4paper]{article}

\usepackage{fancyhdr}
\usepackage{amsmath}
\usepackage{bm}
\usepackage{esint}
\usepackage{amsfonts}
\usepackage{amsthm}
\usepackage{amssymb}
\usepackage{amsbsy}
\usepackage{verbatim}
\usepackage{graphicx}
\usepackage{url}
\usepackage{mathrsfs}
\usepackage[hmargin = 1in,vmargin=.75in]{geometry}
\usepackage{color}
\usepackage{amstext}
\usepackage{stmaryrd}
\usepackage{enumerate}
\usepackage{algpseudocode}
\usepackage{algorithm}
\usepackage{pifont}
\usepackage{prettyref}

\usepackage{subcaption}
\font\msbm=msbm10

\numberwithin{equation}{section}

\theoremstyle{plain}
\newtheorem{theorem}{Theorem}[section]
\newtheorem{lemma}[theorem]{Lemma}
\newtheorem{corollary}[theorem]{Corollary}

\newtheorem{condition}[theorem]{Condition}
\newtheorem{proposition}[theorem]{Proposition}
\newtheorem{definition}{Definition}[section]
\newtheorem{conj}[theorem]{Conjecture}
\newtheorem{remark}[theorem]{Remark}
\def\mathbb#1{\hbox{\msbm{#1}}}

\newcommand{\ba}{\boldsymbol{a}}

\newcommand{\bc}{\boldsymbol{c}}

\newcommand{\br}{\boldsymbol{r}}

\newcommand{\bu}{\boldsymbol{u}}
\newcommand{\bv}{\boldsymbol{v}}
\newcommand{\bw}{\boldsymbol{w}}
\newcommand{\bx}{\boldsymbol{x}}
\newcommand{\by}{\boldsymbol{y}}
\newcommand{\bz}{\boldsymbol{z}}
\newcommand{\bone}{\boldsymbol{1}}
\newcommand{\blambda}{\boldsymbol{\lambda}}
\newcommand{\balpha}{\boldsymbol{\alpha}}

\newcommand{\bgamma}{\boldsymbol{\gamma}}
\newcommand{\bmu}{\boldsymbol{\mu}}

\newcommand{\bSigma}{\boldsymbol{\Sigma}}
\newcommand{\BA}{\boldsymbol{A}}
\newcommand{\BB}{\boldsymbol{B}}

\newcommand{\BD}{\boldsymbol{D}}
\newcommand{\BE}{\boldsymbol{E}}
\newcommand{\BF}{\boldsymbol{F}}
\newcommand{\BG}{\boldsymbol{G}}
\newcommand{\BI}{\boldsymbol{I}}
\newcommand{\BJ}{\boldsymbol{J}}

\newcommand{\BM}{\boldsymbol{M}}

\newcommand{\BQ}{\boldsymbol{Q}}

\newcommand{\BS}{\boldsymbol{S}}

\newcommand{\BU}{\boldsymbol{U}}
\newcommand{\BV}{\boldsymbol{V}}
\newcommand{\BW}{\boldsymbol{W}}
\newcommand{\BX}{\boldsymbol{X}}
\newcommand{\BY}{\boldsymbol{Y}}
\newcommand{\BZ}{\boldsymbol{Z}}

\newcommand{\bphi}{\boldsymbol{\phi}}

\newcommand{\bzero}{\boldsymbol{0}}

\newcommand{\bPsi}{\boldsymbol{\Psi}}
\newcommand{\A}{\mathcal{A}}

\newcommand{\PP}{\mathcal{P}}

\newcommand{\pa}{\partial}

\newcommand{\I}{\boldsymbol{I}}
\newcommand{\RR}{\mathbb{R}}
\newcommand{\lag}{\langle}
\newcommand{\rag}{\rangle}

\newcommand{\TB}{T^{\bot}}

\renewcommand{\Pr}{\mathbb{P}}

\DeclareMathOperator{\Tr}{Tr}

\DeclareMathOperator{\E}{\mathbb{E}}

\DeclareMathOperator{\diag}{diag}

\newenvironment{longtable}[2][c]
   {\table[h]
    \setbox0=\vbox\bgroup \def\caption##1{\omit\gdef\captiontext{##1}}
    \ifx c#1\centering\fi \ifx l#1\raggedright\fi \ifx r#1\raggedleft\fi
    \tabular{#2}
   }
   {\endtabular
    \egroup
    \ifx\captiontext\undefined \else 
       \caption{\captiontext}\par\medskip \global\let\captiontext=\undefined \fi
    \unvbox0 
    \endtable
   }

\newcommand{\vct}[1]{\bm{#1}}
\newcommand{\mtx}[1]{\bm{#1}}
\definecolor{xl}{RGB}{200,50,120}

\begin{document}
\title{\bf When Do Birds of a Feather Flock Together? $k$-Means, Proximity, and Conic Programming}

\author{Xiaodong Li\thanks{Department of Statistics, University of California Davis, Davis CA 95616}, Yang Li\thanks{Department of Mathematics, University of California Davis, Davis CA 95616}, Shuyang Ling\thanks{Courant Institute of Mathematical Sciences and Center for Data Science, New York NY 10012}, Thomas Strohmer$^{\dagger}$, and Ke Wei\thanks{School of Data Science, Fudan University, Shanghai, China, 200433}}

\maketitle

\begin{abstract}
Given a set of data, one central goal is to group them into clusters based on some notion of similarity between the individual objects. One of the most popular and widely-used approaches is $k$-means despite the computational hardness to find its global minimum.
We study and compare the properties of different convex relaxations by relating them to corresponding proximity conditions, an idea originally introduced by Kumar and Kannan. Using conic duality theory, we present an improved proximity condition under which the Peng-Wei relaxation of $k$-means recovers the underlying clusters exactly. Our proximity condition improves upon Kumar and Kannan and is comparable to that of Awashti and Sheffet, where proximity conditions are established for projective $k$-means. In addition, we provide a necessary proximity condition for the exactness of the Peng-Wei relaxation. For the special case of equal cluster sizes, we establish a different and completely localized proximity condition under which the Amini-Levina relaxation yields exact clustering, thereby having addressed an open problem by Awasthi and Sheffet in the balanced case. 

Our framework is not only deterministic and model-free but also comes with a clear geometric meaning which allows for further analysis and generalization.  Moreover, it can be conveniently applied to analyzing various data generative models such as the stochastic ball models and Gaussian mixture models. With this method, we improve the current minimum separation bound for the stochastic ball models and achieve the state-of-the-art results of learning Gaussian mixture models. 
\end{abstract}


\section{Introduction}

$k$-means clustering is one of the most well-known and widely-used clustering methods in unsupervised learning. Given $N$ data points in $\RR^m$, the goal is to partition them into $k$ clusters by minimizing the total squared distance between each data point and the corresponding cluster center. It is a problem related to Voronoi tessellations~\cite{DuFG99}. However,  $k$-means is combinatorial in nature since it  is essentially equivalent to an integer programming problem~\cite{Selim84}. Thus, minimizing the $k$-means objective function turns out to be an NP-hard problem, even if there are only two clusters~\cite{AloiseDPP09} or if the data points are on a $2$D plane~\cite{MahNV09}. 

Despite its hardness, numerous efforts have been made to develop effective and efficient heuristic algorithms to handle the $k$-means problem in practice. 
A famous example is  Lloyd's algorithm~\cite{Lloyd82} which was originally introduced for vector quantization and then became popular in data clustering due to its high efficiency and simplicity of implementation. One of the earliest convergence analyses of  Lloyd's 
algorithm was given by Selim and Ismail~\cite{Selim84}: Under certain conditions, the 
algorithm converges to  a stationary point
within a finite number of iterations but may fail to converge to a local minimum. A smoothed analysis given by Arthur, Manthey and Roglin~\cite{ArthurMR11} shows that the smoothed/expected number of iterations is bounded polynomially by $N$, $k$ and $m$ while the worst-case running time can be $2^{\Omega(N)}$ even for the case when data points are on a plane \cite{Vattani11}.

We are particularly interested in the semidefinite programming (SDP) relaxation for $k$-means by Peng and Wei \cite{PengW07}, who observed that the $k$-means objective function can be written as the inner product between a projection matrix and a distance matrix constructed from the data, and the combinatorial constraints of the projection matrix can be convexified. Thus, whenever the Peng-Wei relaxation produces an output corresponding to a partition of the data set, the $k$-means problem is solved in polynomial time~\cite{Wright1997}. The details of the Peng-Wei relaxation will be explained in \prettyref{sec:prelim}.

Theoretical properties of the Peng-Wei relaxation have also been studied under specific stochastic models in the literature. {\em Minimum separation conditions} were established in~\cite{AwasBCKVW15,IguchiMPV17} to guarantee exact clustering for the stochastic ball models with balanced clusters (i.e., each cluster has the same number of points), while a similar study was conducted in~\cite{MixonVW17} for the Gaussian mixture model.

Despite these efforts, the Peng-Wei relaxation is not yet thoroughly understood. Several fundamental questions of vital importance remain unexplored or require better answers, such as
\emph{
\begin{itemize}
\setlength{\itemsep}{-0.3ex}
\item How do the number of clusters and the data dimension affect the performance of the Peng-Wei relaxation?
\item How does the performance of the Peng-Wei relaxation depend on the balancedness of the cluster sizes and covariance structures within each cluster?
\item Can the global minimum separation condition be localized?
\item Under the special case of equal cluster sizes, does the tighter Amini-Levina relaxation~\cite{AminiL18} improve the Peng-Wei relaxation? If so, in which sense?
\end{itemize}}

The studies in~\cite{AwasBCKVW15,IguchiMPV17,MixonVW17} reveal certain information about the Peng-Wei relaxation based on the assumption of sufficient minimum center separation: guaranteed exact recovery in the case of the stochastic ball model~\cite{AwasBCKVW15,IguchiMPV17} and learning of centers for the Gaussian mixture model~\cite{MixonVW17}. The price to obtain such information, the requirement imposed upon the minimum center separation, is the homogeneity of the criteria forced on all different clusters. In other words, each pair of clusters, regardless of their shapes and cardinalities, must have their centers separated by a uniform distance determined by the {\em entire data set}. 
As a consequence of this ``global'' condition, the effect of an isolated but huge cluster ripples throughout the entire data set by raising the minimum center separation. Thus, a more ``localized'' condition, i.e., a condition on  the  center separation for each pair of clusters that relies largely on local information, is much desired. Such a more localized condition might pave the way to address the aforementioned fundamental questions regarding the Peng-Wei relaxation.

To that end, in this paper we introduce a proximity condition enabling us to relate the pairwise center distances to more localized quantities. Interestingly, it turns out that our proximity condition improves the one in~\cite{KumarK10} and is comparable to that in~\cite{AwasthiS12}, the state-of-the-art proximity conditions in the literature of SVD-based projective $k$-means. Furthermore, under the Amini-Levina relaxation for clusters of equal cardinality, the associated proximity condition becomes even ``fully localized'', as it {\em only} involves information about pairs of clusters.

\subsection{Organization of our paper}
Our paper is organized as follows. In the remainder of this introductory section we present our aforementioned proximity condition, discuss its implication for various stochastic cluster models and briefly compare our results to the state of the art. In Section~\ref{sec:prelim}, we discuss $k$-means and its convex relaxation introduced by Peng and Wei.  In Section~\ref{sec:main}, we show that the Peng-Wei relaxation yields the solution of the $k$-means objective as long as our proximity condition \eqref{eq:prox} is satisfied. A different proximity condition for the exactness of Amini-Levina relaxation is discussed in the same section. In \prettyref{sec:sbm_and_gmm}, we consider the application of our framework to the stochastic ball model and the Gaussian mixture model. 
Numerical simulations that illustrate our theoretical findings are presented in Section~\ref{sec:numerics}. 
All proofs can be found in Sections~\ref{sec:proof4det}--\ref{sec:pf_rand}.

\subsection{Proximity conditions under deterministic models} 
The idea of proximity conditions originates from the work~\cite{KumarK10} by Kumar and Kannan who use a proximity condition to characterize the performance of Lloyd's algorithm with an initialization given by an SVD-based projection under deterministic models. The result is later improved by Awasthi and Sheffet~\cite{AwasthiS12}, who perform a finer analysis and redesign the proximity condition for the same algorithm. To the best of our knowledge, no such type of proximity conditions has been established for the Peng-Wei relaxation so far, and we will fill this gap in this paper. 

Conceptually speaking, our proximity condition can be interpreted as follows:
\begin{center}
\emph{For each pair of clusters, every point is closer to the center of its own cluster, while the bisector hyperplane of the centers keeps all points in the two clusters at a certain distance determined by global information of the data set.}
\end{center}
Roughly speaking, the proximity condition characterizes for each pair of clusters how much closer each point is to the within-cluster center than the cross-cluster center. This is conceptually much more localized than minimum separation, which compares all pairwise center distances to a uniform quantity.

Let us introduce some necessary notation before we proceed to the exact statement of our proximity condition. Given a set of $N$ data points  $\Gamma=\{\bx_l\}_{l=1}^N$ with $k$ mutually disjoint clusters $\Gamma=\sqcup_{a=1}^k\Gamma_a$, we can re-index $\vct{x}_1, \ldots, \vct{x}_N$ according to the clusters: $\Gamma_a=\{\bx_{a,i}\}_{1\leq i\leq n_a}$ for all $1\leq a\leq k$. Denote by $n_a=|\Gamma_a|$ the number of elements in $\Gamma_a$.

Denote the data matrix of the $a$-th cluster by
\[
\BX_a^\top =\begin{bmatrix} \vct{x}_{a, 1} & \ldots & \vct{x}_{a, n_a} \end{bmatrix} \in \RR^{m \times n_a}.
\]
Furthermore, define
\begin{equation*}
\bc_a = \frac{1}{n_a}\sum_{i=1}^{n_a}\bx_{a,i}, \quad \bw_{a,b} = \frac{\bc_b - \bc_a}{\|\bc_b - \bc_a\|},\quad\mbox{and}\quad \overline{\BX}_a = \BX_a - \bone_{n_a} \bc_a^{\top}.
\end{equation*}
In other words,  $\bc_a$ is the sample mean (cluster center) of the $a$-th cluster, $\bw_{a,b}$ is the unit vector pointing from $\bc_a$ to $\bc_b$, and $\overline{\BX}_a$ is the centered data matrix of the $a$-th cluster. Now we are ready to give a mathematical characterization of the proximity condition.
\begin{condition}[\bf Proximity condition]\label{cond:prox}
The partition $\Gamma=\sqcup_{a=1}^k\Gamma_a$ satisfies the proximity condition if for any $a\neq b$, there holds
\begin{equation}\label{eq:prox}
\min_{1\leq i\leq n_a} \left\lag \bx_{a,i} - \frac{\bc_a + \bc_b}{2}, \bw_{b,a}\right\rag > \frac{1}{2}\sqrt{\left(\sum_{l=1}^k\|\overline{\BX}_l\|^2\right)\left( \frac{1}{n_a} + \frac{1}{n_b} \right) }.
\end{equation}
Here, $\|\overline{\BX}_l\|$ is the operator norm of the matrix $\overline{\BX}_l$.
\end{condition}

\begin{figure}[ht!]
\centering
\includegraphics[width=129mm]{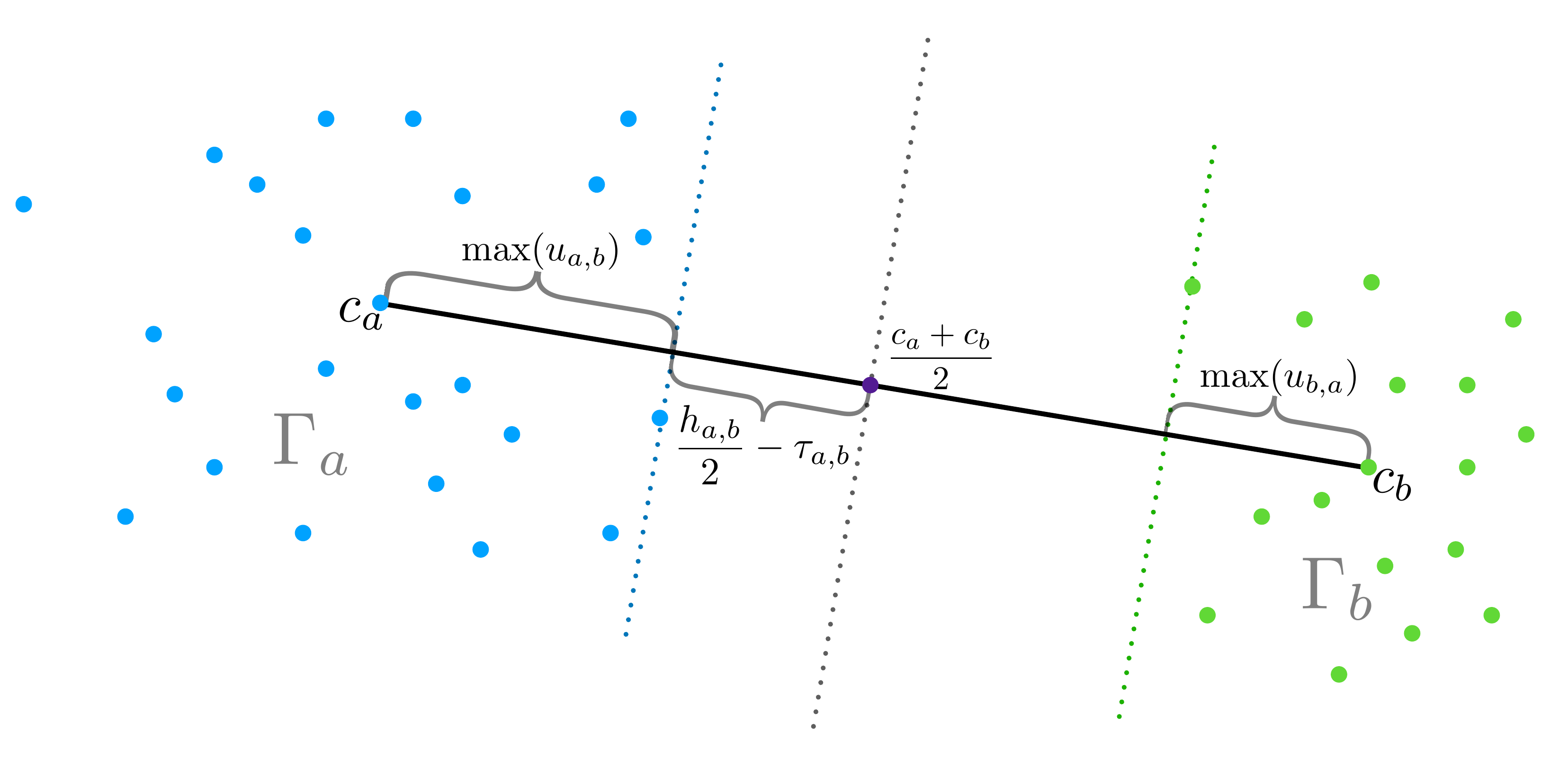}
\caption{Proximity condition: If the partition of data points satisfies the proximity condition, then each pair of clusters $\Gamma_a$ and $\Gamma_b$ can be separated by a plane through the bisector of their sample means $\bc_a$ and $\bc_b$, and the distance between each individual point in those two clusters and the bisector is greater than the right hand side of \eqref{eq:prox}.
}\label{fig:prox}
\end{figure}

The proximity condition has a very intuitive geometric interpretation, see also Figure~\ref{fig:prox}.
Suppose the partition of data points satisfies the proximity condition. Then each pair of clusters $\Gamma_a$ and $\Gamma_b$ can be separated by a plane through the bisector of their sample means $\bc_a$ and $\bc_b$. Moreover, the distance between every point  in those two clusters and the bisector must be greater than the right hand side of \eqref{eq:prox}. This geometric interpretation can be further illustrated by rewriting \eqref{eq:prox}: Denote by $h_{a,b} = \|\bc_a - \bc_b\|$ the distance between the two centers $\bc_a$ and $\bc_b$. Moreover, define
\begin{equation*}
\tau_{a,b} = \max\{ \max(\bu_{a,b}), \max(\bu_{b,a})\} \quad\mbox{where}\quad \bu_{a,b} = \overline{\BX}_a\bw_{a,b} \text{~for~} 1\leq a, b \leq k.
\end{equation*} 
Clearly, $\tau_{a,b}$ is the maximum signed projection distance over all the data points in the clusters $\Gamma_a$ and $\Gamma_b$.
As illustrated in Figure~\ref{fig:prox}, one can easily check that the left hand side of proximity condition~\eqref{eq:prox} is in fact equal to $\frac{1}{2}h_{a,b} - \tau_{a,b}$ which is the shortest distance between the midpoint $\frac{\bc_a + \bc_b}{2}$ and the projections of all the data points in $\Gamma_a$ and $\Gamma_b$ on the line connecting $\bc_a$ and $\bc_b$. This observation gives us the following proposition.
\begin{proposition}\label{prop:prox2}
The proximity condition \eqref{eq:prox} is equivalent to
\begin{equation} \label{eq:prox1}
h_{a,b} > 2\tau_{a,b} + \sqrt{\sum_{l=1}^k\|\overline{\BX}_l\|^2 \left( \frac{1}{n_a} + \frac{1}{n_b} \right) }, \quad  \forall a \neq b.
\end{equation}
\end{proposition}

Besides showing that the proximity condition \eqref{eq:prox} guarantees the exactness of Peng-Wei relaxation, we also obtain a necessary proximity condition. If a deterministic mixture fails to fulfill the necessary condition, exact recovery by the Peng-Wei relaxation is provably impossible. 

Awasthi and Sheffet's has raised an open question in \cite{AwasthiS12}: 
can the pairwise separation condition be fully localized, i.e., depend only on information of the corresponding pair of clusters? We apply the Amini and Levina's relaxation \cite{AminiL18}, originally intended to address the weak assortativity issue in community detection among networks, to convexify the $k$-means problem in the case of balanced clusters. Surprisingly, we end up with a completely localized proximity condition for the exactness of the convex relaxation, thus solving Awasthi and Sheffet's open problem for the balanced case.

Furthermore, beyond the scope of the Peng-Wei relaxation of $k$-means, the proximity condition itself provides an algorithm that can accept answers to the NP-hard $k$-means problem (although it is not able to reject an answer). For a given solution to $k$-means, one can simply check whether the proximity condition holds, and if it does hold, then the solution is provably the unique global minimum. The time cost is proportional to $\mathcal{O}(k N + m^2 N)$. Assuming the number of clusters $k$ and the dimension of data $m$ are fixed, the time complexity is linear in the total number of points $N$, which improves the quasilinear-time algorithm proposed in~\cite{IguchiMPV17} in terms of the time complexity.

\subsection{Comparison to existing proximity conditions in the literature}
As mentioned before, in the literature of projective $k$-means, proximity conditions have been proposed in~\cite{KumarK10} and later improved in~\cite{AwasthiS12}. In this section we compare our proximity conditions with these existing results.
  
Denote $\overline{\BW}=[\overline{\BX}_1^\top, \ldots, \overline{\BX}_k^\top]^\top$. By our notation, the original Kumar-Kannan proximity condition \cite{KumarK10} is equivalent to
\[
h_{a,b} > 2\tau_{a,b} + Ck\left(\frac{1}{\sqrt{n_a}}  + \frac{1}{\sqrt{n_b}}\right)\|\overline{\BW}\|, \quad  \forall a \neq b,
\]
for some large absolute constant $C>0$. The fact that $\max_{1\leq l\leq k}\|\overline{\BX}_l\| \leq \|\overline{\BW}\|$ implies $\sqrt{\sum_{l=1}^k\|\overline{\BX}_l\|^2} \leq \sqrt{k} \|\overline{\BW}\|$. Therefore, our proximity condition~\eqref{eq:prox1} is strictly weaker than the Kumar-Kannan condition by at least a factor of $\sqrt{k}$.

The comparison between \eqref{eq:prox} and the Awasthi-Sheffet conditions in \cite{AwasthiS12} is less straightforward. Theorem 4 therein states that consistent clustering is guaranteed by projective $k$-means plus Lloyd's algorithm as long as 
\begin{equation}
\label{eq:awasthi_min_sep}
h_{a,b} > \max\left\{2\tau_{a,b} + C\left(\frac{1}{\sqrt{n_a}}  + \frac{1}{\sqrt{n_b}}\right)\|\overline{\BW}\|,~~ C\sqrt{k} \left( \frac{1}{\sqrt{n_a}} + \frac{1}{\sqrt{n_b}}\right)\|\overline{\BW}\|\right\} \quad  \forall a \neq b.
\end{equation}

Compared to our proximity condition \eqref{eq:prox}, the second term on the right-hand side of \prettyref{eq:awasthi_min_sep} could be more stringent given the fact $\sqrt{\sum_{l=1}^k\|\overline{\BX}_l\|^2} \leq \sqrt{k} \|\overline{\BW}\|$, whereas the first term is less stringent than ours since
\[
\|\overline{\BW}\|^2 = \|\overline{\BW}^\top \overline{\BW}\| = \left\|\sum_{a=1}^k \overline{\BX}_a^\top \overline{\BX}_a\right\| \leq  \sum_{a=1}^k \left\|\overline{\BX}_a^\top \overline{\BX}_a\right\| = \sum_{a=1}^k \left\|\overline{\BX}_a\right\|^2. 
\]
Therefore, it is fair to say our proximity condition is comparable to the Awasthi-Sheffet condition.

\subsection{Implications under stochastic models}
We should emphasize that in order to prove our main results, we benefit a lot from the existing primal-dual analyses in \cite{AwasBCKVW15,IguchiMPV17}. The major difference between our analysis and~\cite{AwasBCKVW15,IguchiMPV17} is that we aim at deriving proximity conditions under deterministic models rather than establishing minimum separation results under stochastic models. 

However, we are still curious about what minimum separation conditions our proximity condition can yield when applied to both the stochastic ball model and the Gaussian mixture model. Before presenting conditions given by our proximity condition, we first review the state-of-the-art results on both models.
\paragraph{Existing work on the Peng-Wei relaxation:} 
The stochastic ball model can be viewed as a special case of mixture models where the distributions of sample data points are compactly supported on $k$ disjoint unit balls in $\RR^m$. The clusters are balanced and the covariance structure is fairly rigid since all the distributions are assumed to be identical and isotropic. 

Let $\Delta$ be the minimal separation between the cluster centers. In~\cite{AwasBCKVW15}, it is proven that the Peng-Wei relaxation achieves exact recovery provided $\Delta > 2\sqrt{2}(1 + {1}/{\sqrt{m}})$, where the lower bound of $\Delta$ is independent of the number of clusters $k$. Another bound of $\Delta$ is given in~\cite{IguchiMPV17} stating that exact recovery is guaranteed if $\Delta > 2 + {k^2}/{m}$ which is near-optimal in the $m \gg k^2$ regime.

The Gaussian mixture model (GMM) as a stochastic model is more flexible. This model is characterized by its density function which is a weighted sum of the density functions of Gaussian or subgaussian distributions. In~\cite{MixonVW17}, assuming the Gaussian distributions are identical and isotropic, Mixon, Villar and Ward prove that the Peng-Wei relaxation learns the Gaussian centers for balanced clusters when the center separations are required to be above $k \sigma$, where $\sigma \mtx{I}$ is the common covariance of all Gaussian distributions.

\paragraph{Existing work on other algorithms:} 
Clustering Gaussian mixture models has received extensive attention in machine learning and statistics communities. Besides~\cite{MixonVW17}, a lot of progress has been made in developing efficient 
algorithms for this task.
Among them are a family of algorithms here referred to as the projective $k$-means \cite{VemW04,AchM05,KanV09,KumarK10,AwasthiS12,Das99,LuZ16}. 
In general,  the projective $k$-means works in two steps: first project all the data points onto a lower dimensional space usually based on singular value decomposition (SVD), and then classify each point by heuristic methods such as single linkage clustering in~\cite{AchM05} or Lloyd's algorithm in~\cite{AwasthiS12}. 

Vempala and Wang~\cite{VemW04} show that if each pairwise center separation is larger than a quantity determined by the number of clusters $k$, the dimension $m$ and the variances of the clusters, the projective algorithm can classify a mixture of $k$ isotropic Gaussians with high probability. 
Achlioptas and McSherry~\cite{AchM05} show that SVD-based projection followed by single-linkage clustering is able to classify all the sampled data points accurately if the center separation of each pair of clusters is greater than the operator norm of the covariance matrix and the weights of the two clusters  plus a term which depends on the concentration properties of the distributions in the mixture.
The algorithm studied by Kannan and Kumar in~\cite{KumarK10}---the work that first devises the idea of \emph{proximity condition}---also begins with an SVD-based projection and proceeds by Lloyd's algorithm which is initialized by an unspecified near-optimal solution to the $k$-means problem. As stated before, its technical results are improved by Awatshi and Sheffet in~\cite{AwasthiS12}.
Recently, Lu and Zhou~\cite{LuZ16} provide a more detailed estimation of  misclassification rate for each iteration of  Lloyd's algorithm with initialization given by spectral methods~\cite{KanV09}.

\paragraph{Our results:}
We can easily apply the proximity condition to the stochastic ball model and the Gaussian mixture model. The corresponding recovery guarantees are competitive with or improve upon other state-of-the-art results. 
\begin{itemize}
\item
For the stochastic ball model, we show that $\Delta > 2+{\cal O}(\sqrt{{k}/{m}})$ is sufficient to guarantee the exact recovery of the Peng-Wei relaxation, which improves the separation condition $\Delta > 2 + {k^2}/{m}$ in~\cite{IguchiMPV17} when $k$ is large. 
Moreover, our result applies to a broader class of stochastic ball models where each cluster can have a different number of points and may even satisfy a different probability distribution as long as the support of density function is contained within a unit ball. 

\item
For the Gaussian mixture model, we summarize our result for the Peng-Wei relaxation and other state-of-the-art results for both the Peng-Wei relaxation and projective $k$-means in Table~\ref{tab:comp}. It has been shown in~\cite{MixonVW17} that the centers of a Gaussian mixture can be accurately estimated by Peng-Wei relaxation provided the minimal separation is ${\cal O}(k)$. In contrast, our proximity provides a different minimal separation condition ${\cal O}(k^{1/2}+\log^{1/2}{(kN)})$, which is smaller than ${\cal O}(k)$ if $k$ is large and $N$ not too large. Our separation condition is better than \cite{KumarK10} and comparable to \cite{AwasthiS12} for projective $k$-means. Though our bound loses a $k^{1/4}$ factor vis-\`a-vis the one in \cite{VemW04}  for the special case of spherical Gaussian mixtures, we can handle more general Gaussian mixtures where the density functions do not have to be spherical or identical. 
\end{itemize}

\begin{table}[ht!]
\caption{Comparison of results on GMM: 
the separation bound for~\cite{VemW04} only applies to mixtures of isotropic Gaussian distributions and the bound for~\cite{MixonVW17} is used to guarantee learning cluster centers instead of recovering the labels of data points. }\label{tab:comp}
\vspace{-.3cm}
\begin{center}
\begin{tabular}{|l|l|l|l|l|}
\hline
Authors & Separation bounds & Algorithms & Exact & Year \\ \hline
Vempala and Wang~\cite{VemW04} & ${\cal O}(k^{1/4} \log^{1/4}(m))$ & Projective $k$-means & Yes & 2004 \\ \hline
Achlioptas and McSherry~\cite{AchM05} & ${\cal O}(k+k^{1/2} \log^{1/2} N)$ & Projective $k$-means & Yes & 2005 \\ \hline
 Kumar and Kannan~\cite{KumarK10} & ${\cal O}(k(\mbox{polylog} (N)))$ & Projective $k$-means & Yes & 2010 \\ \hline
Awasthi and Sheffet~\cite{AwasthiS12}& ${\cal O}(k^{1/2}(\mbox{polylog} (N)))$ & Projective $k$-means & Yes & 2012\\ \hline
Lu and Zhou~\cite{LuZ16}& ${\cal O}(k^{3/2})$ & Projective $k$-means & No  &2016 \\ \hline
Mixon, Villar, and Ward~\cite{MixonVW17} & ${\cal O}(k)$ & SDP $k$-means & No & 2017 \\ \hline
Our work  & ${\cal O}(k^{1/2} + \log^{1/2}{(kN)})$ & SDP $k$-means & Yes & - \\ \hline
\end{tabular}
\end{center}
\end{table}

\subsection{Notation}

Let $\bone_{\Gamma_a}$ be the indicator vector of $\Gamma_a \subseteq \Gamma$. $\bone_n$ is an $n\times 1$ vector with all entries equal to 1. Given any two real matrices $\BU$ and $\BV$ in $\RR^{m\times n}$, we define the inner product as $\lag \BU, \BV\rag = \Tr(\BU \BV^{\top}) = \sum_{i=1}^m\sum_{j=1}^n U_{ij}V_{ij}$. For a vector $\bv$, $\max(\bv)$ is equal to the largest entry of $\bv$.
We denote $\BZ\geq 0$ if $\BZ$ is a nonnegative matrix, i.e., each entry is nonnegative; $\BZ \succeq 0$ if $\BZ$ is a symmetric positive semi-definite matrix. Besides, we also use the notation listed below throughout the paper.
\begin{longtable}[l]{l l} 
\label{table:notation}
$m$ & Dimension of data \\ 
$k$ & Number of clusters \\ 
$\Gamma$ & Set of $N$ data points in $\RR^m$ \\
$\Gamma_a$ & The $a$-th cluster \\
$N$ & Total number of data points \\ 
$n_a$ & Number of points in the $a$-th cluster \\ 
${\cal S}^N$ & Set of $N\times N$ symmetric matrices \\
${\cal S}^N_+$ & Set of $N\times N$ positive semi-definite matrices\\ 
$\RR^{N\times N}_+$ & Set of $N\times N$ nonnegative matrices\\ 
$\BW$ & Data matrix of all $N$ data points \\
$\BX_a$ & Data matrix of the $a$-th cluster \\
$\overline{\BX}_a$ & Centered data matrix of the $a$-th cluster \\
$\BD$ & Squared distance matrix\\
$\BX$ & Ground-truth solution to the SDP relaxation of $k$-means \\
$\BY^{(a,b)}$ & Submatrix of any $N \times N$ matrix $\BY$ given by $\{ y_{s,t} \}_{s \in \Gamma_a, t \in \Gamma_b}$ \\
$\bx_{a,i}$ & The $i$-th data point in the $a$-th cluster \\
$\bmu_a$ & Population mean of the $a$-th cluster in a generative model \\
$\bc_a$ & Sample mean of the $a$-th cluster \\
$\bw_{a,b}$ & Unit vector pointing from $\bc_a$ to $\bc_b$ \\
$\bu_{a,b}$ & Signed projection distance given by $\bu_{a,b} = \overline{\BX}_a \bw_{a,b}$ \\
$h_{a,b}$ & Distance between $\bc_a$ and $\bc_b$ \\
$\tau_{a,b}$ & Maximum signed projection distance determined by $\bu_{a,b}$ and $\bu_{b,a}$ \\
\end{longtable}


\section{$k$-means and the Peng-Wei relaxation} 
\label{sec:prelim}
In this section, we briefly review the formulation of $k$-means and its SDP relaxation introduced by Peng and Wei~\cite{PengW07}. Let $\Gamma=\{\bx_l\}_{l=1}^N$ be a set of $N$ data points in $\RR^m$. $k$-means attempts to divide $\Gamma$ into $k$ disjoint clusters by seeking a solution to the following minimization problem:
\begin{equation*}
\min_{\{\Gamma_a\}_{a=1}^k} \min_{\{\bgamma_a\}_{a =1}^k}~\sum_{a=1}^k \sum_{l\in \Gamma_a} \left\| \bx_l - \bgamma_a  \right\|^2,
\end{equation*}
where $\{\Gamma_a\}_{a=1}^k$ form a partition of $\Gamma$ (i.e., $\sqcup_{a=1}^k\Gamma_a=\Gamma$ and $\Gamma_a\sqcap\Gamma_b=\emptyset$ if $a\neq b$). For any given partition $\{\Gamma_a\}_{a=1}^k$, choosing $\bgamma_a$ as the centroid $\bgamma_a = \bc_a = \frac{1}{|\Gamma_a|} \sum_{j\in\Gamma_a} \bx_j~ (a = 1, \ldots, k)$ minimizes the objective function. Therefore, the $k$-means problem is equivalent to:
\begin{equation}\label{eq:km}
\min_{\{\Gamma_a\}_{a=1}^k} ~\sum_{a=1}^k \sum_{l\in \Gamma_a} \left\| \bx_l - \bc_a  \right\|^2,
\end{equation}

Given an arbitrary partition $\{\Gamma_a\}_{a=1}^k$ of $\Gamma$, let $\bone_{\Gamma_a}~ (a = 1, \ldots, k)$ be the indicator function of the $a$-th cluster. That is, 
\begin{align*}
\bone_{\Gamma_a}(l) = \begin{cases}
1 &\mbox{if }l\in\Gamma_a,\\
0 & \mbox{otherwise}.
\end{cases}
\end{align*}
A simple calculation  can reveal that 
\begin{equation*}
  \frac{1}{|\Gamma_a|} \sum_{l\in\Gamma_a, s\in\Gamma_a}  \|\bx_l - \bx_s\|^2 = 2\sum_{l\in\Gamma_a}  \|\bx_l - \bgamma_a\|^2
\end{equation*}
and hence,
\begin{align*}
\sum_{a=1}^k \sum_{l\in\Gamma_a} \left\| \bx_l - \bmu_a\right\|^2  &= \frac{1}{2} \sum_{a=1}^k \frac{1}{|\Gamma_a|} \sum_{l\in\Gamma_a, s\in\Gamma_a} \|\bx_l - \bx_s\|^2\\
 &=\frac{1}{2} \sum_{a=1}^k\frac{1}{|\Gamma_a|}\lag \bone_{\Gamma_a}\bone_{\Gamma_a}^\top, \BD\rag,
\end{align*}
where $\BD\in\RR^{N\times N}$ is the distance matrix with the $(l,s)$-th entry being given by $\BD_{l,s}=\|\bx_l-\bx_s\|^2$. Therefore, we can rewrite the $k$-means problem as 
\begin{align}\label{eq:km2}
\begin{split}
\min \quad &  \lag \BZ, \BD\rag \\
\mbox{s.t.} \quad &  \BZ =\sum_{a=1}^k\frac{1}{|\Gamma_a|}\bone_{\Gamma_a}\bone_{\Gamma_a}^\top \mbox{ with }\sqcup_{a=1}^k\Gamma_a=\Gamma\mbox{ and }\Gamma_a\sqcap\Gamma_b=\emptyset \mbox{ for }a\neq b.
\end{split}
\end{align}
It is self-evident that \eqref{eq:km2} is a non-convex problem due to the combinatorial nature of the feasible set. Indeed, \eqref{eq:km2} is an NP-hard problem~\cite{AloiseDPP09}. Despite this, it can be easily verified that $\BZ=\sum_{a=1}^k\frac{1}{|\Gamma_a|}\bone_{\Gamma_a}\bone_{\Gamma_a}^\top$ satisfies the following four properties: 
\begin{equation*}
\BZ \succeq 0, \quad \BZ \geq 0,  \quad \BZ \bone_N = \bone_N,\quad \Tr(\BZ) = k.
\end{equation*} 
Replacing the constraint in \eqref{eq:km2} by the above four properties leads to the SDP relaxation of $k$-means introduced by Peng and Wei in~\cite{PengW07},
\begin{align}\label{eq:sdp}
\begin{split}
\min \quad &  \lag \BZ, \BD\rag \\
\mbox{s.t.} \quad &  \BZ \succeq 0, \quad \BZ \geq 0,  \quad \BZ \bone_N = \bone_N,\quad \Tr(\BZ) = k,
\end{split}
\end{align}
which will be the focus of this paper.

The Peng-Wei relaxation is a convex problem and can be solved in polynomial time using the interior-point method~\cite{Wright1997}. We denote by $\BX$ the optimal solution to the Peng-Wei relaxation. Clearly, every feasible point of \eqref{eq:km2} is also feasible for \eqref{eq:sdp}; so once the optimal solution to \eqref{eq:sdp} has the form $\BX=\sum_{a=1}^k\frac{1}{|\Gamma_a|}\bone_{\Gamma_a}\bone_{\Gamma_a}^\top$, it must be an optimal solution to the $k$-means problem. Therefore, the question of central importance is:
\begin{quote}
\centering\large\em
When is the solution to \eqref{eq:sdp} of the form $\BX=\sum_{a=1}^k\frac{1}{|\Gamma_a|}\bone_{\Gamma_a}\bone_{\Gamma_a}^\top$?
\end{quote}


\section{Exact recovery guarantees} \label{sec:main}

\subsection{Exact clustering and proximity conditions}
\label{sec:proximity_and_determ}
In a nutshell our following main theorem states that the proximity condition \eqref{eq:prox} implies the exactness of the Peng-Wei relaxation \eqref{eq:sdp}:
\begin{theorem}[\bf Main theorem]\label{thm:main}
Suppose the partition $\{\Gamma_a\}_{a=1}^k$ obeys the proximity condition \eqref{eq:prox}.  Then the  minimizer of the Peng-Wei relaxation~\eqref{eq:sdp} is unique and given by 
$
\BX=\sum_{a=1}^k\frac{1}{|\Gamma_a|}\bone_{\Gamma_a}\bone_{\Gamma_a}^\top.
$ 
\end{theorem}
Since the global minimum of~\eqref{eq:sdp} is always smaller than that of~\eqref{eq:km}, Theorem~\ref{thm:main} implies that the proximity condition provides a simple algorithm that is able to accept answers to the $k$-means problem.
\begin{corollary} [\bf Algorithm accepting answers to $k$-means]
If a partition $\Gamma=\sqcup_{a=1}^k\Gamma_a$ satisfies the proximity condition \eqref{cond:prox}, then it is the unique global minimum to the $k$-means objective function.
\end{corollary}
Note that each data point $\bx_{a,i}$ appears $k-1$ times on the left hand side of \eqref{cond:prox}, and it takes ${\cal O}(m^2 n_a)$ amount of time to compute each matrix operator norm using the Golub-Reisch SVD algorithm \cite{GolubV96}. Thus, the time cost to examine the proximity condition is proportional to ${\cal O}(k N + m^2 N)$.

To the best of our knowledge, $k$-means problem has not been shown in NP or not. The proximity condition does not change this fact. We want to emphasize that the polynomial time examination of the proximity condition \eqref{cond:prox} does not imply that an answer to the $k$-means problem can be verified in polynomial time since it does not accept all correct answers. A different approach that leverages the dual certificate associated with the Peng-Wei relaxation to test under certain conditions  the optimality of a candidate $k$-means solution can be found in~\cite{IguchiMPV17}. The algorithm proposed in~\cite{IguchiMPV17} tests the optimality of a candidate solution in quasilinear time. Hence, our method improves the time complexity by a logarithmic factor.

While the main theorem provides a sufficient condition for the Peng-Wei relaxation to exactly recover a given partition, the following theorem gives a necessary condition.

\begin{theorem}[\bf Necessary condition] \label{thm:lower}
Suppose  $\BX=\sum_{a=1}^k\frac{1}{|\Gamma_a|}\bone_{\Gamma_a}\bone_{\Gamma_a}^\top$ is a global minimum of \eqref{eq:sdp}. 
Then the partition $\{\Gamma_a\}_{a=1}^k$ must satisfy
\begin{equation} \label{eq:necessary}
h_{a,b} \geq \tau_{a,b} + \sqrt{\tau_{a,b}^2 + \max_t \| \overline{\BX}_t \|^2 \left( \frac{1}{n_a} + \frac{1}{n_b} \right) }, \quad \forall a \neq b.
\end{equation}
\end{theorem}

Notice that as long as $\BX$ is a solution to \eqref{eq:sdp}, $\{\Gamma_a\}_{a=1}^k$ must be a global minimum to the $k$-means. In other words, it is harder for a deterministic mixture to be exactly recovered by the Peng-Wei relaxation than being the global minimum to the $k$-means. It remains unclear whether this necessary condition (Theorem~\ref{thm:lower}) is only necessary for the Peng-Wei relaxation or is necessary for the $k$-means itself as well.

\subsection{Balanced case: Amini-Levina relaxation and proximity condition}
\label{sec:balanced}
One special case of interest is the balanced case where each cluster has the same number of points, i.e. $|\Gamma_1| = \ldots = |\Gamma_k| = n$. We have seen in Section~\ref{sec:prelim} that the $k$-means problem can be rewritten as \eqref{eq:km2}:
\begin{align}
\begin{split}
\min \quad &  \lag \BZ, \BD\rag \\
\mbox{s.t.} \quad &  \BZ =\sum_{a=1}^k\frac{1}{|\Gamma_a|}\bone_{\Gamma_a}\bone_{\Gamma_a}^\top \mbox{ with }\sqcup_{a=1}^k\Gamma_a=\Gamma\mbox{ and }\Gamma_a\sqcap\Gamma_b=\emptyset \mbox{ for }a\neq b.
\end{split}
\end{align}
With the balanced assumption, i.e., the cardinalities of all clusters being the same, it is easy to verify that $\BZ = \sum_{a=1}^k \frac{1}{n} \bone_{\Gamma_a}\bone_{\Gamma_a}^\top$ obeys the following four constraints:
\begin{equation*}
\BZ \succeq 0, \quad \BZ \geq 0,  \quad \BZ \bone_N = \bone_N,\quad \diag(\BZ) = \frac{1}{n} \bone_N.
\end{equation*} 
This leads to the Amini-Levina relaxation of $k$-means, which was first introduced in \cite{AminiL18} for community detection under balanced case in order to address the weak assortativity issue:
\begin{align}\label{eq:sdp-bal} 
\begin{split}
\min \quad &  \lag \BZ, \BD\rag \\
\mbox{s.t.} \quad &  \BZ \succeq 0, \quad \BZ \geq 0,  \quad \BZ \bone_N = \bone_N,\quad \diag(\BZ) = \frac{1}{n} \bone_N.
\end{split}
\end{align}
As with the analyses on the Peng-Wei relaxation, once the optimal solution to \eqref{eq:sdp-bal} takes the form $\BX = \sum_{a=1}^k \frac1 n \bone_{\Gamma_a}\bone_{\Gamma_a}^\top$, the Amini-Levina relaxation gives an optimal solution to the $k$-means problem with balanced assumption. Once again, we ask the same question for Peng and Wei's relaxation:
{\em 
When is the solution to \eqref{eq:sdp-bal} of the form $\BX=\sum_{a=1}^k\frac{1}{n}\bone_{\Gamma_a}\bone_{\Gamma_a}^\top$?}

Unsurprisingly, the answer is another proximity condition specially tailored for Amini and Levina's relaxation.
\begin{condition}[\bf Proximity condition for balanced clusters]\label{cond:prox-bal}
A partition $\Gamma=\sqcup_{a=1}^k\Gamma_a$ with $|\Gamma_1| = \ldots = |\Gamma_k| = n$ satisfies the proximity condition for balanced clusters if for any $a\neq b$, there holds
\begin{equation}\label{eq:prox-bal}
\min_{1\leq i\leq n_a} \left\lag \bx_{a,i} - \frac{\bc_a + \bc_b}{2}, \bw_{b,a}\right\rag > 
\sqrt{\frac{k}{4n}\left(\|\overline{\mtx{X}}_a\|^2 + \|\overline{\mtx{X}}_b\|^2\right)}.
\end{equation}
\end{condition}
Similar to the general case, the proximity condition for balanced clusters also has an equivalent formulation:
\begin{equation} \label{eq:sdp-bal2}
h_{a,b} > 2\tau_{a,b} + \sqrt{\frac{k}{n}\left(\|\overline{\mtx{X}}_a\|^2 + \|\overline{\mtx{X}}_b\|^2\right)}.
\end{equation}

\begin{theorem}[\bf Exact recovery for balanced clusters] \label{thm:main-bal}
Suppose the partition $\{\Gamma_a\}_{a=1}^k$ with $|\Gamma_1| = \ldots = |\Gamma_k| = n$ obeys the proximity condition for balanced clusters \eqref{eq:prox-bal}.  Then the  minimizer of the Amini-Levina relaxation~\eqref{eq:sdp-bal} is unique and given by 
$
\BX=\sum_{a=1}^k\frac{1}{n}\bone_{\Gamma_a}\bone_{\Gamma_a}^\top.
$ Therefore, the partition $\{\Gamma_a\}_{a=1}^k$ can be recovered exactly by the Amini-Levina relaxation.
\end{theorem}

Compared with the proximity condition for Peng and Wei's relaxation \eqref{eq:prox}, the proximity condition for Amini and Levina's relaxation distinguishes itself by decoupling the clusters in the sense that each of the $k(k-1)$ inequalities in \eqref{eq:prox-bal} only depends on the two clusters involved in the inequality. In the case of balanced clusters, this immediately solves the open question posed by Awasthi and Sheffet \cite{AwasthiS12}, which asks if such a proximity condition exists.

The completely localized proximity condition is particularly meaningful when there are a few abnormal clusters whose covariance matrices are huge in matrix operator norm, but at the same time being away from all the other clusters. In this case, the proximity condition for Amini and Levina's relaxation has far better chance than that for Peng and Wei's relaxation to detect a reasonable partition of the data set. Figure~\ref{fig:prox_bal_ex} provides such an example.

\begin{figure}[h]
\centering
\includegraphics[width = 100mm]{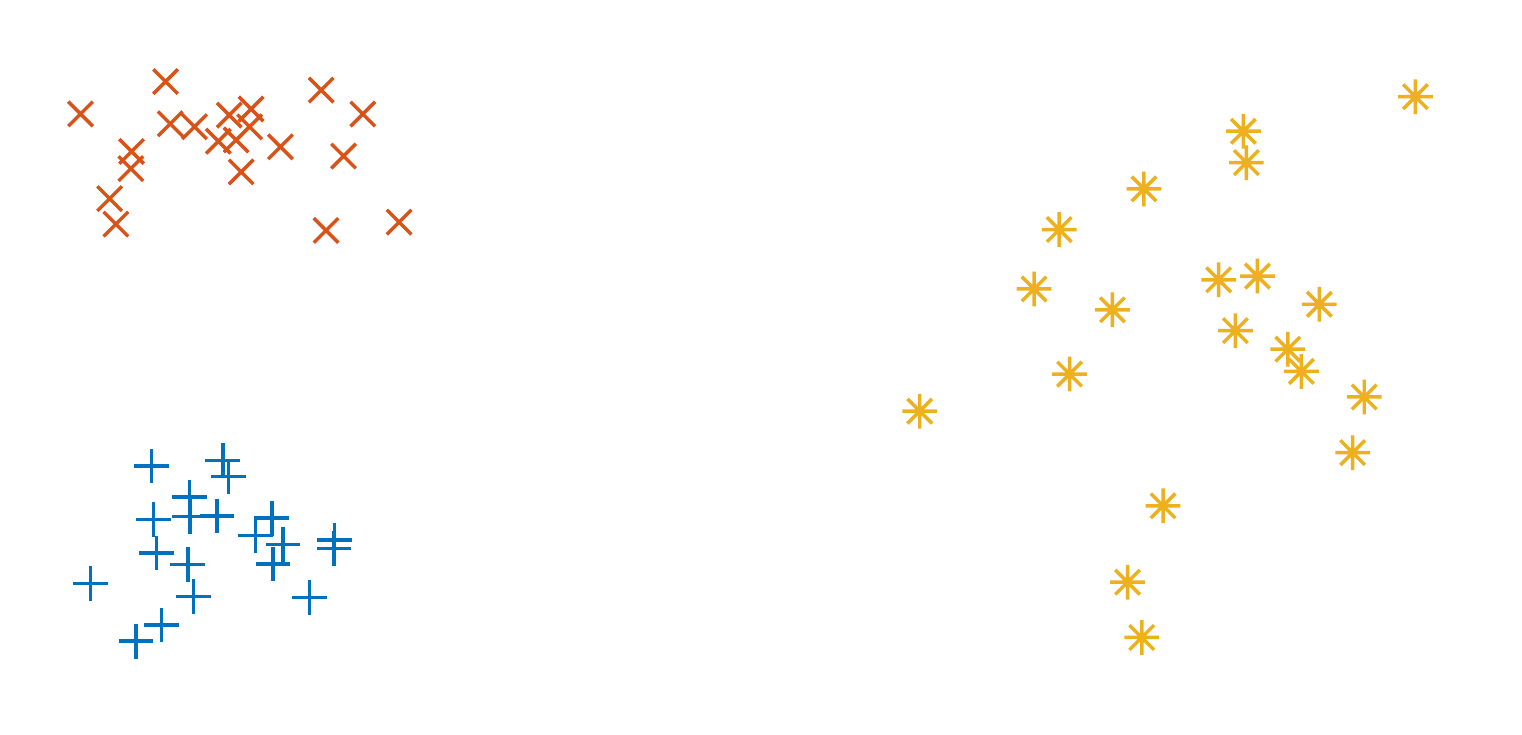}
\caption{An example of three clusters in the plane. Each contains 20 points. The proximity for the general case \eqref{eq:prox} fails for this instance. However, the proximity condition for balanced clusters \eqref{eq:prox-bal} is satisfied and hence ensures the partition is optimal to the $k$-means problem with balanced assumption.}
\label{fig:prox_bal_ex}
\end{figure}

Analogously, we can also prove a necessary condition for the Amini-Levina relaxation, which can be compared with Theorem~\ref{thm:lower} for the general case.
\begin{theorem}[\bf Necessary condition for balanced clusters] \label{thm:lower-bal}
Suppose  $\BX=\sum_{a=1}^k\frac{1}{|\Gamma_a|}\bone_{\Gamma_a}\bone_{\Gamma_a}^\top$ is a global minimum of \eqref{eq:sdp-bal}. 
Then the partition $\{\Gamma_a\}_{a=1}^k$ must satisfy
\begin{equation} \label{eq:necessary-bal}
h_{a,b} \geq \tau_{a,b} + \sqrt{\tau_{a,b}^2 +  \frac{1}{n} \left( \| \overline{\BX}_a \|^2 + \| \overline{\BX}_b \|^2 \right)}, \quad \forall a \neq b.
\end{equation}
\end{theorem}

\section{Results under random models}
\label{sec:sbm_and_gmm}

Next we apply the proximity condition \eqref{eq:prox} to data sets generated from the generalized stochastic ball model and the Gaussian mixture model, respectively. We first give a formal definition for each model and then present  the minimal separation condition which is sufficient to guarantee the exact recovery of underlying clusters by the Peng-Wei relaxation. The minimal separation conditions are established by verifying the proximity condition~\eqref{cond:prox} for those two random models. For proofs, see Sections~\ref{sec:sbm} and~\ref{sec:gmm}.

\subsection{Stochastic ball model}

The definition of generalized stochastic ball model is given as follows where we only assume the support of the density function is contained in the unit ball of $\RR^m$ for all clusters. 
\begin{definition}[\bf Generalized stochastic ball model]
Let $\{\bmu_a\}_{a=1}^k$ be a set of $k$ deterministic vectors in $\mathbb{R}^m$. For each $1 \leq a \leq k$, $\mathcal{D}_a$ is a distribution supported on the unit ball of $\RR^m$ with a covariance matrix $\bSigma_a$ and $\{\br_{a,i}\}_{i=1}^{n_a}$ are i.i.d. zero-mean random vectors drawn from the distribution $\mathcal{D}_a$. The $a$-th cluster is formed by $\{\bx_{a,i}\}_{i=1}^{n_a}$, where $\bx_{a,i}=\bmu_a+\br_{a,i}$ for $1\leq i\leq n_a$. 
\end{definition}

\begin{corollary}\label{cor:rbm}
Denote $\sigma^2_{\max} = \max_{1\leq a\leq k}\|\bSigma_a\|$, $N = \sum_{a=1}^k n_a$, $w_{\min} = \frac{1}{N}\min_{1\leq a\leq k}n_a$, and $\Delta = \min_{a\neq b}\|\bmu_a - \bmu_b\|.$
For the generalized stochastic ball model, we draw $n_a$ points from the $a$-th ball for each $1 \leq a \leq k$.  The Peng-Wei relaxation achieves exact recovery with probability at least $ 1 - N^{-\gamma}$ if $N \geq \frac{4}{w_{\min}} \log(4 kmN^{\gamma})$ and
\begin{equation}\label{eq:rbm-prox-main}
\Delta \geq 2 + \sqrt{\frac{2}{w_{\min}}} \sigma_{\max} + 7 \sqrt{\frac{t}{w_{\min}}},
\end{equation}
where $t = \sqrt{\frac{4\log(4kmN^{\gamma}) }{Nw_{\min}}}$ and $\gamma > 0$.
In particular, if $n_a = n$ for all $a$, $w_{\min} = \frac{1}{k}$  and each $\mathcal{D}_a$ is a uniform distribution over the unit ball of $\RR^m$, then~\eqref{eq:rbm-prox-main} can be simplified to
\begin{equation*}
\Delta \geq 2 + \sqrt{\frac{2k}{m+2}} + 7\sqrt{tk}
\end{equation*}
by noting that $\sigma^2_{\max} = \|\bSigma_a\| = \frac{1}{m+2}.$
\end{corollary}

\begin{remark}
As the number of data points $N$ goes to infinity  provided $k$ and $w_{\min}$ are fixed,  the value of $t = \sqrt{\frac{4\log(4kmN^{\gamma}) }{Nw_{\min}}}$ vanishes. So asymptotically the minimal separation condition reduces to $\Delta > 2 + \sqrt{\frac{2k}{m+2}}$ when $n_a = n$ and $\bSigma_a = \frac{1}{m+2}\I_m$. Note that we only assume that the distribution is supported on the unit ball, so rotation-invariant distributions which are assumed in~\cite{IguchiMPV17,iguchi2015tightness} are also included. Compared with the result in~\cite{IguchiMPV17,iguchi2015tightness} where $\Delta > 2 + \frac{k^2}{m}$ is required, we have achieved a better bound when $k$ is large.
\end{remark}

We can also apply the necessary lower bound (Theorem~\ref{thm:lower}) to the generalized stochastic ball model. To illustrate this, let us study a special case where the following Corollary holds.
\begin{corollary} \label{cor:lower-rbm}
For the generalized ball model, if for all $1 \leq a \leq k$ we have $n_a = n$, then with high probability, the Peng-Wei relaxation fails to achieve exact recovery provided that $N$ is large enough and
\begin{equation*}
\Delta <   1 + \sqrt{1 + 2 \sigma_{\max}^2}.
\end{equation*}
If for any $a$, ${\cal D}_a$ is the uniform distribution over the unit ball, the bound becomes
\begin{equation*}
\Delta <   1 + \sqrt{1 + \frac{2}{m+2}}.
\end{equation*}
\end{corollary}

\subsection{Gaussian mixture model}
The definition of Gaussian mixture model is given below, followed by the minimal separation condition for the exactness of the Peng-Wei relaxation.
\begin{definition}[\bf Gaussian mixture model]
Consider a mixture of $k$ Gaussian distributions $\mathcal{N}(\bmu_a, \bSigma_a)$ in $\RR^m$ with a set of weights $\{w_a\}_{a=1}^k$ obeying $w_a\geq 0$ and $\sum_{a=1}^kw_a=1$. The probability density function of this mixture model is
\begin{equation*}
p(\bx) = \sum_{a=1}^k w_a p_\mathcal{N}(\bx; \bmu_a, \bSigma_a), \quad \bx \in \RR^m,
\end{equation*}
where $p_\mathcal{N}(\bx; \bmu_a, \bSigma_a)$ is the probability density function of the Gaussian distribution $\mathcal{N}(\bmu_a, \bSigma_a)$. 
\end{definition}

\begin{corollary}
\label{cor:gmm}
Denote $\sigma^2_{\max} = \max_{1\leq a\leq k}\{\|\bSigma_a\|\}$, $w_{\min} = \min_{1\leq a\leq k}\{w_a\}$ and $\Delta = \min_{a\neq b}\|\bmu_a - \bmu_b\|$.
For the  Gaussian mixture model, the Peng-Wei relaxation achieves exact recovery with probability at least $1 - 6N^{-1}$ if
\begin{equation*}
\Delta  \geq \sigma_{\max} \left( \frac{2}{\sqrt{w_{\min}}} + 4\sqrt{2} \log^{1/2} (k N^{2}) +  q(N;m,k,w_{\min}) \right),
\end{equation*}
where $q(N;m,k,w_{\min}) = o(1)$ if $N \gg m^2 k^2 \log(k)/ w_{\min}$.
In particular, if $n_a = n$ and $\bSigma_a = \I_m$  for all $1\leq a\leq k$, then the above condition reduces to
\begin{equation*}
\Delta \geq 2\sqrt{k} + 4\sqrt{2} \log^{1/2} (k N^{2}) + q(N;m,k,1/k),
\end{equation*}
and $q(N;m,k,1/k) = o(1)$ if $N \gg m^2 k^3 \log(k)$.

\end{corollary}

\section{Numerical experiments}\label{sec:numerics}

Consider applying the Peng-Wei relaxation to the generalized stochastic ball model. When the total number of the data points $N$ becomes large enough, the parameter $t$ vanishes and the sufficient lower bound predicted by Corollary~\ref{cor:rbm} as in \eqref{eq:rbm-prox-main} becomes
\begin{equation}
\label{eq:our-bound}
\Delta \geq 2 + \sigma_{\max} \sqrt{\frac{2}{w_{\min}}}.
\end{equation}
The state-of-the-art bound for the stochastic ball model proved in ~\cite{AwasBCKVW15,IguchiMPV17} is
\begin{equation}
\label{eq:sota-bound} 
\Delta > \min \left\{ 2\sqrt{2} \left(1+\frac{1}{\sqrt{m}}\right), 2+\frac{k^2}{m} \right\}.
\end{equation}

The exact phase transition bound, above which exact recovery can be achieved by the Peng-Wei relaxation of $k$-means, is smaller than both of the above sufficient lower bounds. As one would expect, the actual lower bound is hard to find in practice. The major difficulty occurs when the number of clusters $k$ is greater than 2. In this case, when creating an instance of the stochastic ball model with prescribed minimal separation distance $\Delta$, there are infinitely many possible ways to place the centers and this cannot be resolved by translation, rotation, and scaling. To address this, we investigate the worst case where centers are packed as compactly as possible while points in each cluster are chosen in the most scattered way. We have a better chance finding a more accurate lower bound under this arrangement.

Three instructive centroidal geometries, the geometries formed by the locations of the centers, are considered, and we call them circle-shaped geometry, line-shaped geometry, and hive-shaped geometry respectively. Centers are packed compactly under these shapes, especially the hive-shaped geometry. We can rescale the three geometries to change the minimal separation distance $\Delta$. An illustration of these geometries formed by the locations of the centers is shown in Figure~\ref{fig:center-shape}.

\begin{figure}[h]
\centering
\includegraphics[width = 150mm]{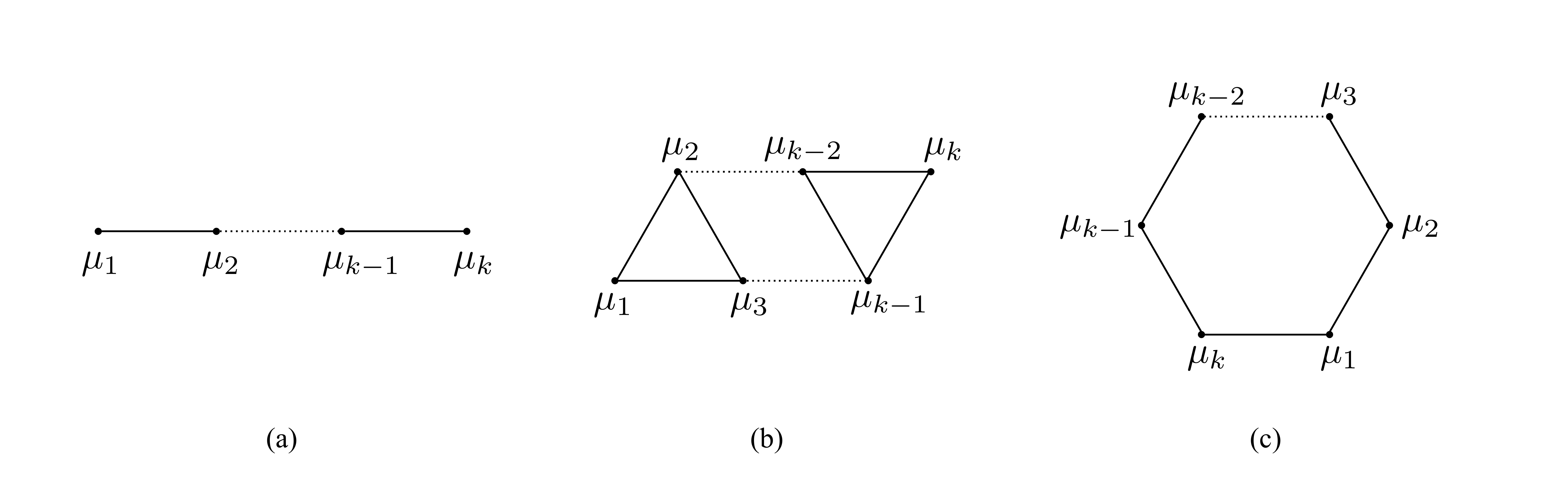}
 
\caption{Illustration of three instructive centroidal geometries. The minimal separation $\Delta$ is the distance between two adjacent centers. Our bound refers to \eqref{eq:our-bound} with parameters calculated for the given distribution. The state-of-the-art bound \eqref{eq:sota-bound} is the bound proved by ~\cite{AwasBCKVW15,IguchiMPV17}}
\label{fig:center-shape}
\end{figure}

We let the number of data points in each cluster be $n_a = 100$. Hence, the total number of points $N = 100k$. As a result, $w_{\min} = 1/k$. These $n_a$ points are equispaced points on the unit circle centered at $\bmu_a$. The data points are chosen in this way since it maximizes the variance. Because the data is isotropic and the variance is equal to $1$, we have $\sigma_{\max} =  1/\sqrt{m} = 1/\sqrt{2}$.

\begin{figure}[h]
\centering
\includegraphics[width = 100mm]{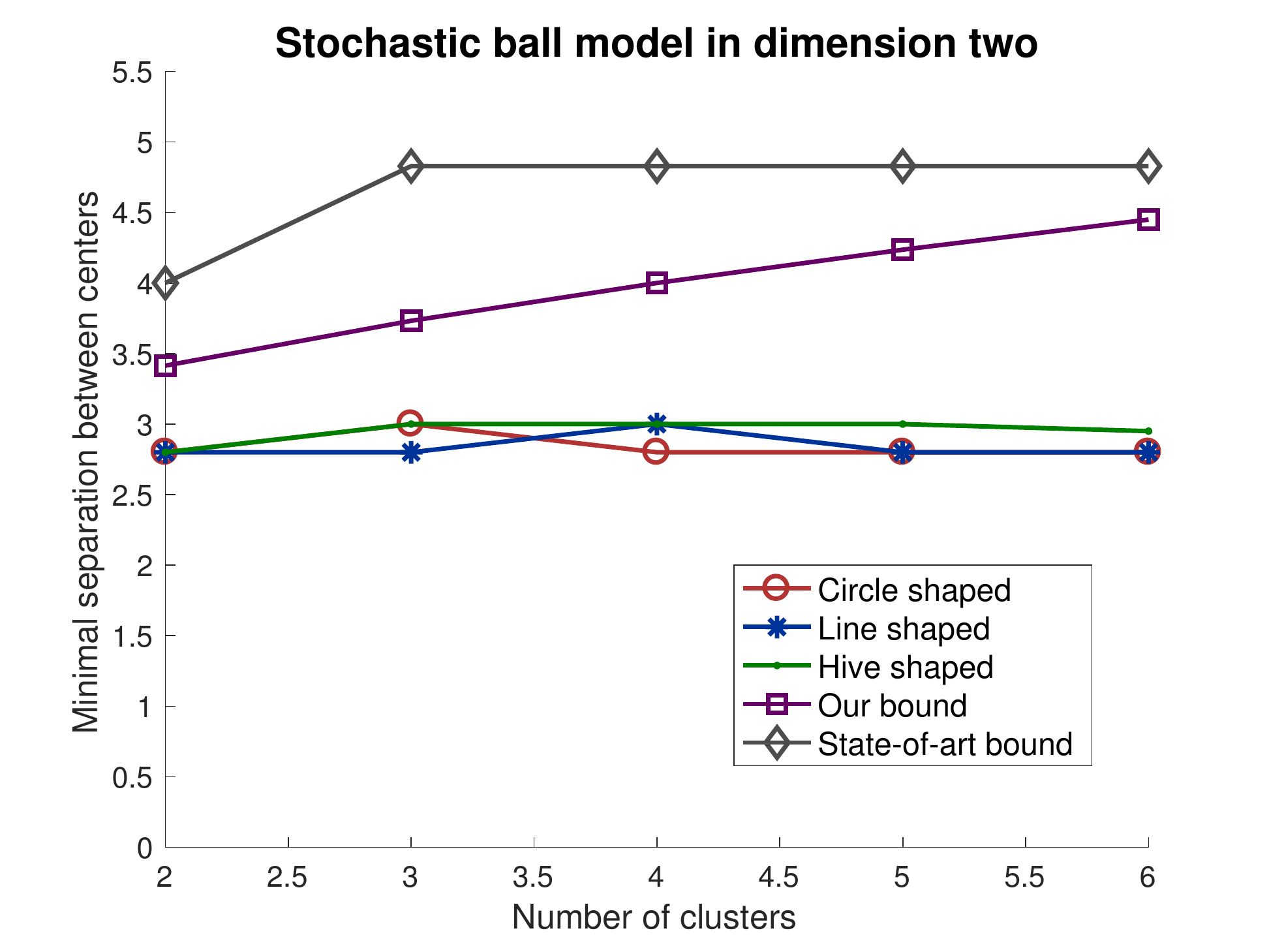}
\caption{Numerical experiment on the stochastic ball model with dimension 2 and number of clusters varying from 2 to 6. The sufficient lower bound here is the bound proved in Corollary~\ref{cor:rbm}. The Peng-Wei relaxation (SDP) is solved by SDPNAL+v0.5 (beta) \cite{yang2015sdpnal+, zhao2010newton}.}
\label{fig:sbm_dim2}
\end{figure}

For $k$ and $m$ chosen above, we can see that our bound is an improvement to the state-of-the-art result. Overall, it is still a meaningful addition to the state-of-the-art result. Nevertheless, it is not yet tight. Figure~\ref{fig:sbm_dim2} shows that the actual lower bound is almost independent of the parameter $k$, while our theory still relies on the assumption that $\Delta \geq 2 + \mathcal{O} ( \sqrt{k/m} )$.

Another parameter that may affect the bound is the dimension $m$. To reveal dependence of the bound on the dimension, we fix the number of clusters $k$ to be $2$ and let the dimension $m$ vary between $2$ and $10$. The center separation $\Delta$ is chosen among $100$ equispaced number between $2$ and $4$. The number of points in each cluster $n_a$ is equal to $25 \times 2^{m-1}$, so there are $N = 50 \times 2^{m-1}$ in total. The distribution $\mathcal{D}_a$ for each ball is the uniform distribution on the unit sphere centered at $\bmu_a$. For any fixed pair of $m$ and $\Delta$, we generate $20$ instances of the stochastic ball model.

\begin{figure}[h]
\centering
\includegraphics[width = 100mm]{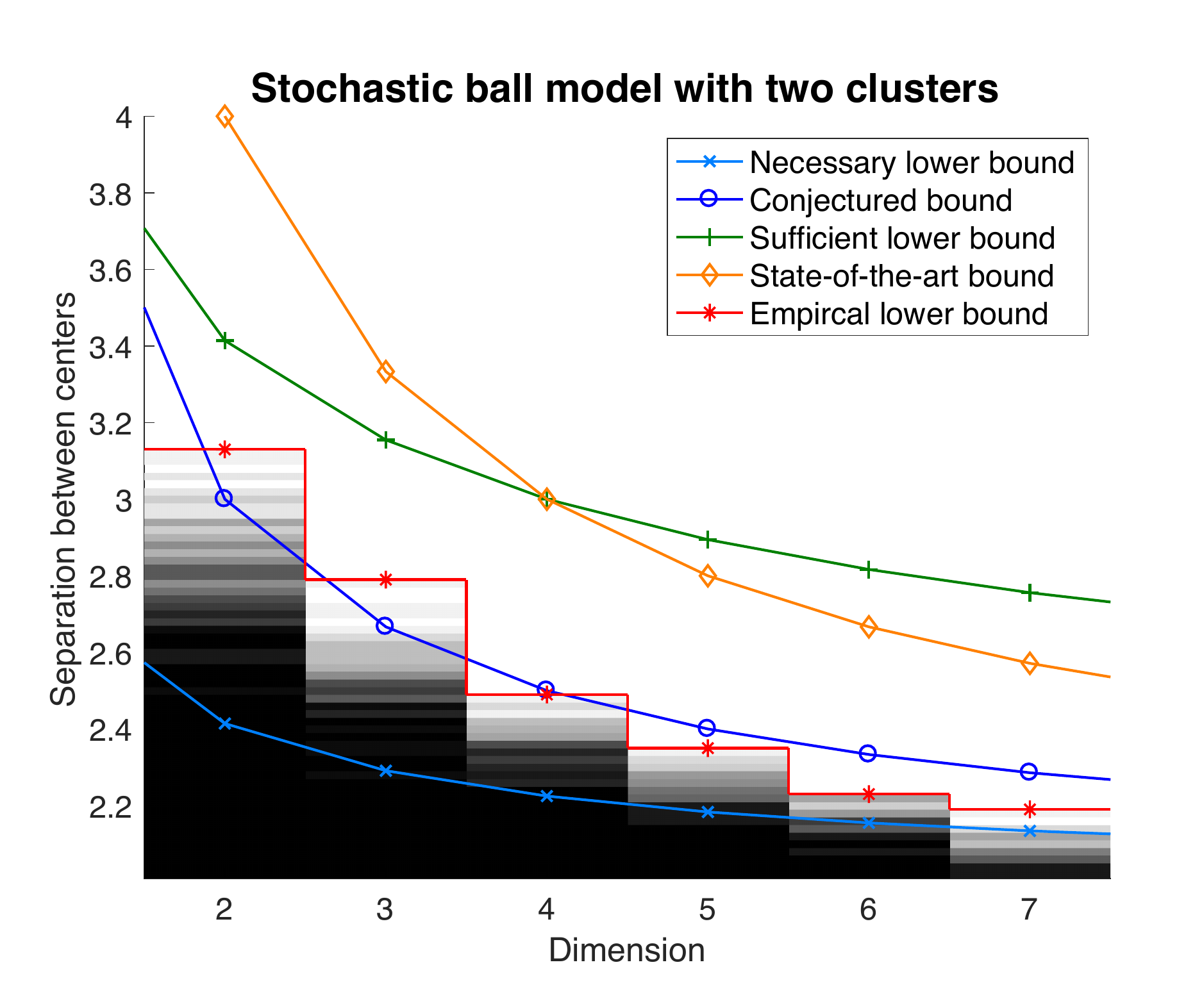}
\caption{Numerical experiment on the stochastic ball model with 2 clusters and dimension varying from 2 to 7. For given dimension and separation, the lighter the color is, the higher the probability of success is. The sufficient lower bound here is the bound given by Corollary~\ref{cor:rbm}, while the necessary lower bound is obtained by applying Theorem~\ref{thm:lower} directly to the stochastic ball model, which is $1 + \sqrt{1+2/m}$ in this case. Being constrained by computational resources, we are not able to sample more points in higher dimension since the time cost is prohibitive. This infers that the right half of the empirical lower bound is potentially smaller than the exact phase transition bound, which is what we are trying to approximate in this experiment. The Peng-Wei relaxation (SDP) is executed via SDPNAL+v0.5 (beta) \cite{yang2015sdpnal+, zhao2010newton}.}
\label{fig:sbm_k2}
\end{figure}

From Figure~\ref{fig:sbm_k2}, it is evident that neither our bound nor the state-of-the-art bound is tight. The blue line, which represents the bound $\Delta \geq 2 + \frac{2}{m}$, fits our empircal result the best. Based on the observation of dependence between the empirical lower bound and the parameters $k$ and $m$ as in Figure~\ref{fig:sbm_dim2} and \ref{fig:sbm_k2} , we formulate a conjecture as stated below.

\begin{conj}\label{conj:rbm}
For a mixture generated by the generalized stochastic ball model, the Peng-Wei relaxation achieves exact recovery with high probability if
\begin{equation}
\Delta \geq 2 + \mathcal{O} \left( \frac{1}{m} \right),
\end{equation}
provided that the total number of points $N$ is large enough.
\end{conj}
After the completion of this manuscript, a semidefinite relaxation based on graph cuts has been proposed in~\cite{LingS18} to overcome the performance limits of Peng-Wei relaxation, which provides a new alternative way to learn the stochastic ball models.


\section{Proofs for Section~\ref{sec:proximity_and_determ}} \label{sec:proof4det}

We will  prove the main theorem and related results under the proximity condition given in Proposition~\ref{prop:prox2}. The proof for the equivalence of the two proximity conditions is presented  at the end of this section. The key ingredient in the proof of the main theorem is to construct a dual variable to certify the optimality of the desired  solution $\BX=\sum_{a=1}^k\frac{1}{|\Gamma_a|}\bone_{\Gamma_a}\bone_{\Gamma_a}^\top$   based on the conic duality theorem in convex optimization~\cite{BenN01}.

\subsection{Conic duality}
We first rewrite~\eqref{eq:sdp} as a cone program in standard form which naturally leads to its dual formulation. 
Noting that $\BZ$ is a symmetric variable, the Peng-Wei relaxation of $k$-means~\eqref{eq:sdp} is equivalent to the following optimization problem:
\begin{align} \label{eq:prime}
\begin{split}
\min \quad &  \lag \BZ, \BD\rag \\
\mbox{s.t.} \quad &  \BZ \succeq \bzero, \quad \BZ \geq \bzero,  \quad \frac{1}{2} (\BZ + \BZ^\top) \bone_N = \bone_N,\quad \Tr(\BZ) = k.
\end{split}
\end{align}

Let ${\cal K} = {\cal S}^{N}_+ \cap \RR^{N\times N}_{+}$, the intersection of two self-dual cones: the positive semi-definite cone ${\cal S}^{N}_+$ and the nonnegative cone $\RR^{N\times N}_{+}$. By definition, it is a pointed\footnote{${\cal K}$ is pointed if for $\BZ\in {\cal K}$ and $-\BZ\in {\cal K}$, $\BZ$ must be $\bzero$, see Chapter 2 in~\cite{BenN01}.} and closed convex cone with a nonempty interior. 
Moreover, its dual cone\footnote{The dual cone of ${\cal K}$ is defined as $\{\BW : \lag \BW, \BZ\rag\geq 0, \forall \BZ\in {\cal K}\}$; in particular, there holds $({\cal K}^*)^* = {\cal K}.$} is  given by ${\cal K}^* = {\cal S}^{N}_+ + \RR^{N\times N}_{+} = \{\BB + \BQ: \BB\geq \bzero, \BQ\succeq \bzero\}$.
Let $\A$ be a linear map $\A$ from  ${\cal S}^N$ to $\RR^{N+1}$ defined as follows:
\begin{equation*}
\A(\BZ) :  \quad \BZ \rightarrow \begin{bmatrix}
\lag \BZ, \I_N\rag  \\
\frac{1}{2}(\BZ+\BZ^\top)\bone_N) 
\end{bmatrix}.
\end{equation*}
We can express \eqref{eq:prime} in the form of a standard cone program,
\begin{equation}\label{eq:standard}
\min \quad \lag \BZ, \BD\rag, \quad \mbox{s.t.} \quad \A(\BZ) = \begin{bmatrix} k \\ \bone_N  \end{bmatrix}, \quad \BZ \in {\cal K}.
\end{equation}
Thus, using the standard derivation in Lagrangian duality theory~\cite{BoydV04}, the dual problem of~\eqref{eq:prime} can be easily obtained and given by
\begin{equation}\label{eq:dual}
\max \quad -kz - \lag \balpha, \bone_N\rag, \quad \mbox{s.t.} \quad \BD + \A^*\left( \blambda\right) \in {\cal K}^*,
\end{equation}
where $\blambda = \begin{bmatrix} z \\ \balpha \end{bmatrix}\in\RR^{N+1}$ is the dual variable with respect to the affine constraints and 
\begin{equation}\label{def:Aadj}
\A^*(\blambda) := \frac{1}{2}(\balpha\bone_N^{\top} + \bone_N \balpha^{\top}) + z\I_N
\end{equation}
 is the adjoint operator of $\A$ under the canonical inner product over $\RR^{N\times N}$.

\subsection{Optimality condition}

This subsection presents a  necessary and sufficient condition for $\BX=\sum_{a=1}^k\frac{1}{|\Gamma_a|}\bone_{\Gamma_a}\bone_{\Gamma_a}^\top$ to be the global minimum of the Peng-Wei relaxation. The result is summarized in  Proposition~\ref{thm:opt_no_unq}, which follows from the complementary slackness in the conic duality theory. Moreover, a stronger sufficient condition has been established for the uniqueness of $\BX$ in Proposition~\ref{lem:optimal}.

\begin{theorem}[\bf Conic Duality Theorem, Theorem 2.4.1 in~\cite{BenN01}]\label{thm:conic} 
There hold:
\begin{enumerate}[1.]
\item If the primal problem is strictly feasible and bounded below, then the dual program is solvable\footnote{The primal problem or dual problem is solvable if it is feasible, bounded and the optimal value is attained.} and the optimal values of the primal/dual problems are equal to each other;
\item If the dual problem is strictly feasible and bounded above, then the primal program is solvable and the optimal values of the primal/dual problems are equal to each other;
\item Assume either the primal problem or the dual problem is bounded and strictly feasible. Then $(\BZ, \blambda)$ is a pair of primal/dual optimum if and only if either the duality gap is zero or the complementary slackness holds.
\end{enumerate}
\end{theorem}

The following lemma, tailored to~\eqref{eq:prime} and~\eqref{eq:dual}, simply follows from the strict feasibility of~\eqref{eq:prime} or~\eqref{eq:dual} and Theorem~\ref{thm:conic}. 
\begin{lemma} \label{lem:slater}
Both primal/dual problems~\eqref{eq:prime} and~\eqref{eq:dual} are strictly feasible and bounded below/above. Therefore, they are  are solvable (so the optimal values are attained). Moreover, $(\BX, \blambda)$ is a pair of primal/dual optima if and only  if the complementary slackness holds: $\lag \BD + \A^*(\blambda), \BX\rag = 0$ where $\BD +\A^*(\blambda)\in{\cal K}^*$.
\end{lemma}
\begin{proof}
Consider $\widetilde{\BZ} = \frac{1-\lambda}{N} \bone_N \bone_N^\top + \lambda \I_N$, where $\lambda = \frac{k-1}{N-1} > 0$ for $k \geq 2$. Note that $\widetilde{\BZ} \succeq \lambda \I_N \succ \bzero$ and $\widetilde{\BZ} \geq \frac{1-\lambda}{N} \bone_N \bone_N^\top > \bzero$. So $\widetilde{\BZ}$ is in the interior of ${\cal K}$. It is also easy to verify that $\widetilde{\BZ}$ satisfies the other two equality constraints. This shows \eqref{eq:prime} is strictly feasible. 
In addition, we can see that the objective function in~\eqref{eq:prime} is also nonnegative since both $\BZ$ and $\BD$ are entrywise nonnegative.
In conclusion, the primal problem is strictly feasible and bounded below by $0$. 

Note that $\BJ_{N\times N} = \bone_N \bone_N^\top$ is a strictly positive symmetric matrix. For the dual problem~\eqref{eq:dual}, we can take  $\balpha = \bzero$ and  let $z$ be a sufficiently large positive number such that
\begin{equation*}
\BD + \A^*(\blambda) = \underbrace{\BJ_{N\times N}}_{\text{a positive matrix}} + \underbrace{\left( \BD + z\I_N - \BJ_{N\times N} \right)}_{\text{a positive definite matrix}}
\end{equation*} 
is in the interior of ${\cal K}^*$. Hence, the dual program is also strictly feasible. Its optimal value is bounded above because it is always smaller than the optimal value of the primal problem. 
 
Therefore, the application of Theorem~\ref{thm:conic} implies that $(\BX, \blambda)$ is a pair of primal/dual optima if and only if the complementary slackness holds, i.e., $\lag \BD + \A^*(\blambda), \BX\rag = 0$ where $\BD + \A^*(\blambda)\in{\cal K}^*$ and $\BX\in{\cal K}.$
 \end{proof}
\begin{remark}
The complementary slackness is indeed equivalent to the zero duality gap since  the optimal values of both problems are attained and there holds
\begin{equation*}
\lag\BD, \BX\rag = -\lag \A^*(\blambda), \BX\rag= -\lag \blambda, \A(\BX)\rag = -\left\lag \blambda, \begin{bmatrix}k \\\bone_N\end{bmatrix}\right\rag = -kz - \lag \balpha, \bone_N\rag.
\end{equation*}
\end{remark}

In the following lemma, we will derive a more explicit expression for complementary slackness which will be used in the analysis later.
By definition of ${\cal K}^*$, the matrix $\BD + \A^*(\blambda)$ must be in the form of
\begin{equation}\label{def:BQ}
\BD + \A^*(\blambda) = \BB + \BQ,
\end{equation}
where $\BB \geq \bzero$, $\BQ\succeq 0$ and both of them are symmetric.

\begin{lemma}\label{lem:comp_slack}
The complementary slackness $\lag \BD + \A^*(\blambda), \BX\rag = 0$ is equivalent to
\begin{equation}\label{eq:BQdual}
\BB^{(a,a)}= \bzero~\mbox{~for all~}1\leq a\leq k, \quad\mbox{and}\quad \BQ\BX = \BX\BQ =\bzero,
\end{equation}
where $\BB\geq \bzero$ and $\BQ\succeq \bzero$ obeys \eqref{def:BQ} for some $\blambda$.
It follows immediately that $\BQ^{(a,b)}\bone_{n_b} = \bzero$ for $1\leq a, b\leq k.$ 
Moreover,~\eqref{eq:BQdual} implies that the dual variable $\blambda=\begin{bmatrix} z\\ \balpha \end{bmatrix}$ satisfies
\begin{equation}\label{eq:alphaa}
\balpha_{a} = -\frac{2}{n_a}  \BD^{(a,a)} \bone_{n_a}  + \frac{1}{n_a^2} \lag \BD^{(a,a)}, \BJ_{n_a \times n_a} \rag \bone_{n_a}  - \frac{z}{n_a} \bone_{n_a},
\end{equation}
where $\balpha_a$ is the $a$-th block of $\balpha$ given by $\{ \alpha_i \}_{i \in \Gamma_a}$.
\end{lemma}

\begin{proof}
It suffices to prove~\eqref{eq:BQdual} from $\lag \BD +\A^*(\blambda), \BX\rag = 0$ since the other direction is trivial. Note that the complementary slackness is equivalent to $\lag \BB +\BQ, \BX\rag = 0$ for some $\BB\geq \bzero$ and $\BQ\succeq \bzero$. Since $\BX\geq \bzero$ and $\BX\succeq 0$, it follows that $\lag \BB, \BX\rag = \lag \BQ, \BX\rag = 0.$ From $\lag \BB, \BX\rag = 0$ and $\BB\geq \bzero$, we have
\begin{equation*}
\lag \BB^{(a,a)}, \BJ_{n_a\times n_a}\rag = 0 \Longleftrightarrow \BB^{(a,a)} = \bzero
\end{equation*} 
where $\BX^{(a,a)} = \BJ_{n_a\times n_a}.$
Since both $\BX$ and $\BQ$ are positive semi-definite matrices, we have 
\begin{equation*}
0 = \lag \BX, \BQ\rag = \Tr(\BX\BQ) = \| \BX^{1/2}\BQ^{1/2}\|_F^2,
\end{equation*}
which gives $\BQ^{1/2}\BX^{1/2}  = \BX^{1/2}\BQ^{1/2}  = \bzero$ and in turn implies $\BQ\BX = \BX\BQ = \bzero$. 

Now we proceed to derive~\eqref{eq:alphaa}.
Following from $\BQ^{(a,a)} \bone_{n_a} = \bzero$ and $\BB^{(a,a)} = \bzero$, we obtain
\begin{align*}
\begin{cases}
\BQ^{(a,a)} \bone_{n_a} = \BD^{(a,a)}\bone_{n_a} + \frac{1}{2}( n_a \balpha_a +  \balpha_a^{\top}\bone_{n_a} \bone_{n_a}) + z\bone_{n_a} = \bzero,\\
\bone_{n_a}^\top \BQ^{(a,a)} \bone_{n_a} = \bone_{n_a}^\top \BD^{(a,a)} \bone_{n_a} + n_a \balpha_a^{\top}\bone_{n_a} + n_a z= \bzero,
\end{cases}
\end{align*}
where $\BQ = \BD + \frac{1}{2}(\balpha\bone_N^{\top} +\bone_N\balpha^{\top}) + z\I_N -\BB$ follows from $\BB+\BQ=\BD + \A^*(\blambda)$ and the definition of $\A^*$, see \eqref{def:BQ} and \eqref{def:Aadj}.
From the second equation above, we get $\balpha_a^{\top}\bone_{n_a} = -\frac{1}{n_a} \bone_{n_a}^\top \BD^{(a,a)} \bone_{n_a} - z$. Substituting it into the first one gives
\begin{equation*}
\balpha_a = \frac{1}{n_a}\left(-2\BD^{(a,a)}\bone_{n_a} - \balpha_a^{\top}\bone_{n_a}\bone_{n_a} - 2z\bone_{n_a}\right) = \frac{1}{n_a}\left(-2\BD^{(a,a)}\bone_{n_a} +  \frac{1}{n_a} \bone_{n_a} \bone_{n_a}^\top \BD^{(a,a)} \bone_{n_a}  - z\bone_{n_a}\right),
\end{equation*}
which completes the proof.
 \end{proof}

Because of~\eqref{eq:alphaa}, the effective dual variables are only $z$ and $\BB^{(a,b)}$ with $a \neq b$ since $\balpha$ can be fully represented by a function of $z$ if the complementary slackness holds, and
plugging $\balpha$ back into the expression of $\BQ$ in~\eqref{def:BQ} gives
\begin{equation} \label{eq:Q}
\BQ = z (\I_N-\BE) + \BM - \BB,
\end{equation}
where 
\begin{align} \label{def:M} 
\begin{cases}
\BE^{(a,b)} & = \frac{1}{2}\left( \frac{1}{n_a} + \frac{1}{n_b}\right) \BJ_{n_a \times n_b}, \\
\BM^{(a,b)} & = \BD^{(a, b)} - \left( \frac{1}{n_a} \BD^{(a, a)}\BJ_{n_a\times n_b}  + \frac{1}{n_b} \BJ_{n_a\times n_b} \BD^{(b, b)}\right) \\
  &  \qquad + \frac{1}{2} \left( \frac{1}{n_a^2}\langle \BD^{(a, a)}, \BJ_{n_a\times n_a} \rangle + \frac{1}{n_b^2}\langle \BD^{(b, b)}, \BJ_{n_b\times n_b} \rangle \right) \BJ_{n_a\times n_b}.
\end{cases}
\end{align}
In particular, if $a = b$, 
\begin{align}\label{def:M_special} 
\begin{cases}
\BE^{(a,a)} & = \frac{1}{n_a} \BJ_{n_a \times n_a},\\
\BM^{(a,a)} & = \left(\I_{n_a} - \frac{1}{n_a} \BJ_{n_a\times n_a} \right)\BD^{(a,a)} \left(\I_{n_a} - \frac{1}{n_a} \BJ_{n_a\times n_a} \right).
\end{cases}
\end{align}

On the other hand, if $\BB\geq 0$, $\BB^{(a,a)}= \bzero~\mbox{~for all~}1\leq a\leq k$, and $\BQ\succeq 0$ has the form of \eqref{eq:Q}, then one can easily verify that  $\BQ\BX = 0$ since $\lag\BQ,\BX\rag = 0$, and $\BB+\BQ=\BD+\A^*(\blambda)$ for $z$ in \eqref{eq:Q} and $\balpha$ in \eqref{eq:alphaa}. Therefore, Lemma~\ref{lem:comp_slack} implies that $\BX$ is a global minimizer of \eqref{eq:prime}.

In summary, we have established a necessary and sufficient condition for $\BX$ to be a global minimizer of the Peng-Wei relaxation of $k$-means.

\begin{proposition}[\bf Optimality condition] \label{thm:opt_no_unq}
Any feasible pair of $\BQ \succeq \bzero$ and $\BB \geq \bzero$ where $\BQ$ has the form of \eqref{eq:Q} and $\BB^{(a,a)} = \bzero$ for all $1 \leq a \leq k$, certifies $\BX$ to be a global minimum of \eqref{eq:prime}. Conversely, if $\BX$ is a global minimum of \eqref{eq:prime}, then such a pair of $(\BQ, \BB)$ $(\text{or }(z, \BB))$ must exist.
\end{proposition}

The optimality condition we have established is essentially equivalent to that of \cite{iguchi2015tightness}. However, we use conic duality theory in \cite{BenN01} to show the strong duality holds, and both primal/dual solutions exist for Peng-Wei relaxation by constructing a Slater's constraint qualification. This lays the foundation to derive the necessary condition for the tightness of Peng-Wei relaxation, which is not fully addressed in \cite{iguchi2015tightness}.

In other words, the optimality condition in Proposition~\ref{thm:opt_no_unq} is not strong enough to guarantee that $\BX$ is a unique solution to \eqref{eq:prime}. The following proposition provides a sufficient condition for the uniqueness of $\BX$ by imposing a stricter condition on $\BB$.

\begin{proposition}[\bf A sufficient condition for the uniqueness of global minimum] \label{lem:optimal}

Any feasible pair of $\BQ \succeq \bzero$ and $\BB \geq \bzero$, where $\BQ$ has the form of \eqref{eq:Q}, $\BB^{(a,a)} = \bzero$ for all $1\leq a \leq k$, and $\BB^{(a,b)} > \bzero$ for all $a \neq b$, certifies $\BX$ to be a unique global minimum of \eqref{eq:prime}.
\end{proposition}

\begin{proof}
Proposition~\ref{thm:opt_no_unq} implies $\BX$ is a global minimum of \eqref{eq:prime}.
Let $\widetilde{\BX}\in\RR^{N\times N}$ be an arbitrary feasible solution satisfying $\widetilde{\BX}\bone_N = \bone_N$, $\Tr(\widetilde{\BX}) = k$, $\widetilde{\BX}\succeq 0$ and $\widetilde{\BX}\geq 0$. 
We will prove $\BX$ is a unique solution by showing that if $\widetilde{\BX} \neq \BX$, there holds
\begin{equation*}
\lag \BD, \widetilde{\BX} - \BX\rag > 0.
\end{equation*}

We start with $\lag \BQ, \widetilde{\BX}-\BX\rag$. Since $\BQ\succeq 0$, $\widetilde{\BX}\succeq 0$, and $\lag \BQ, \BX\rag = 0$, it follows that
\begin{equation*}
\lag \BQ, \widetilde{\BX}- \BX\rag = \lag \BQ, \widetilde{\BX}\rag \geq 0. 
\end{equation*}
By the definition of $\BQ$, and the fact $\widetilde{\BX}\bone_N = \BX\bone_N= \bone_N$ and $\Tr(\widetilde{\BX}) = \Tr(\BX) = k$, there holds,
\begin{equation*}
\lag \BQ, \widetilde{\BX}- \BX\rag = \lag \BD, \widetilde{\BX}- \BX\rag - \lag \BB, \widetilde{\BX}- \BX\rag \geq 0.
\end{equation*}
Since the supports of $\BB$ and $\BX$ are disjoint, one has $\lag \BB, \BX\rag = 0.$
Therefore, in order to show $\lag \BD, \widetilde{\BX} - \BX\rag > 0$, it suffices to prove that $\lag\BB, \widetilde{\BX} \rag > 0$, which will be done by contradiction.

Suppose  $\lag \BB, \widetilde{\BX}\rag = \sum_{a\neq b}\lag \BB^{(a,b)}, \widetilde{\BX}^{(a,b)}\rag = 0$. Then we have $\widetilde{\BX}^{(a,b)} = 0$ which follows from $\BB^{(a,b)} > 0$ for all $a\neq b$ and $\widetilde{\BX}\geq 0.$ Therefore, the support of $\widetilde{\BX}$ must be the same as that of $\BX$.
Note that $\widetilde{\BX}$ is a positive semi-definite  
matrix which satisfies $\widetilde{\BX} \bone_N = \bone_N$ and $\Tr(\widetilde{\BX}) = k$.  So for any $1 \leq a \leq k$, $\widetilde{\BX}^{(a,a)} \bone_{n_a} = \bone_{n_a}$. This means that $1$ is an eigenvalue of $\widetilde{\BX}$ with multiplicity at least $k$. Since all the eigenvalues of $\widetilde{\BX}$ are nonnegative and their sum is equal to $\Tr(\widetilde{\BX}) = k$, $\widetilde{\BX}$ has only $k$ nonzero eigenvalues and all of them are 1. Thus, each $\widetilde{\BX}^{(a,a)}$ is a rank one matrix. It follow that  $\widetilde{\BX}^{(a,a)} = \frac{1}{n_a} \bone_{n_a} \bone_{n_a}^{\top} = \BX^{(a,a)}$ since $\widetilde{\BX}^{(a,a)} \bone_{n_a} = \bone_{n_a}$ and $\widetilde{\BX}^{(a,a)}$ is symmetric. This contradicts the assumption $\widetilde{\BX} \neq \BX.$

 \end{proof}

\subsection{Sufficient condition for dual certificate} \label{sec:suff_dual_cert}
We will further reduce the sufficient condition in Proposition~\ref{lem:optimal} to one that will be used in the construction of the dual certificate. 
As suggested by that proposition, we need to find a number $z \in \RR$ and a symmetric  matrix $\BB \in \RR^{N \times N}_+$ such that the following sufficient condition holds:
\begin{equation}\label{eq:optimal}
\BQ\succeq \bzero,  \quad \BB^{(a,b)} > \bzero, \quad \BB^{(a,a)} = \bzero \quad \forall a\neq b,
\end{equation}
where $\BQ$  is given in \eqref{eq:Q}. As a result $\BQ$, satisfies $\BQ \BX =\BX\BQ= \bzero$ automatically.

In order to present our final sufficient optimality condition, we first introduce two linear subspaces. 
Note that $\BX$ is clearly a projection matrix satisfying $\BX^2 = \BX$. Let $T$ and $\TB$ be two linear subspaces in $\RR^{N\times N}$ defined as
\begin{align*}
T &  = \{ \BX \BY + \BY\BX - \BX\BY\BX: \BY\in\RR^{N\times N} \}, \\
\TB &  =  \{ (\I_N - \BX) \BY (\I_N - \BX): \BY\in \RR^{N\times N} \}.
\end{align*}
Denote by $\PP_T: \RR^{N \times N} \rightarrow T$ and $\PP_{\TB}: \RR^{N \times N} \rightarrow \TB$ the corresponding projection operators. We use subscripts to denote projections, for example letting $\PP_T(\BB)=\BB_T $ and $\PP_{T^\perp}(\BB)=\BB_{T^\perp} $. 
For any $\BZ\in\RR^{N\times N}$, it can be easily verified that the $(a,b)$-th block of $\BZ_T$ and $\BZ_{\TB}$ are
\begin{align}
& \BZ^{(a,b)}_{T} = \frac{1}{n_a}\BJ_{n_a \times n_a}\BZ^{(a,b)} + \frac{1}{n_b} \BZ^{(a,b)} \BJ_{n_b \times n_b} - \frac{1}{n_an_b}\BJ_{n_a \times n_a}\BZ^{(a,b)}\BJ_{n_b \times n_b}, \\
& \BZ^{(a,b)}_{\TB} = \left( \I_{n_a} - \frac{1}{n_a}\BJ_{n_a \times n_a}\right) \BZ^{(a,b)}\left( \I_{n_b} - \frac{1}{n_b}\BJ_{n_b \times n_b}\right).
\end{align}

\begin{proposition} \label{prop:eqv_opt}
The optimality condition with uniqueness in \eqref{eq:optimal} is equivalent to
\begin{align}\label{eq:eqv_opt}
\begin{cases}
z {\cal P}_{\TB} (\I_N) +\BM_{\TB} - \BB_{\TB} \succeq  \bzero, \\
 \BM_T^{(a,b)} - \BB^{(a,b)}_T - \frac{z(n_a + n_b)}{2n_an_b}\BJ_{n_a\times n_b} = \bzero, \quad \forall a \neq b, \\
\BB^{(a,b)} = (\BB^{(b,a)})^{\top}, \quad \BB^{(a,a)} = \bzero, \quad  \BB^{(a,b)} >  \bzero, \quad \forall \, a \neq b.
\end{cases}
\end{align}
\end{proposition}

\begin{proof}
We first show that \eqref{eq:optimal} implies \eqref{eq:eqv_opt}, and then show the other direction. 
\paragraph{$\eqref{eq:optimal}\Longrightarrow\eqref{eq:eqv_opt}$:}
Noting that $\BE\in T$, ${\cal P} (\I_N) = \BI_N - \BX$ and $\BQ$ has the form of \eqref{eq:Q}, the projection of $\BQ$ on $\TB$ is given by
\begin{equation*}
\BQ_{\TB} = (\I_N - \BX)\BQ(\I_N - \BX)=  z(\I_N-\BX) + \BM_{\TB} - \BB_{\TB}\succeq \bzero
\end{equation*}
which gives the first expression in~\eqref{eq:eqv_opt}. For the second one in~\eqref{eq:eqv_opt},  we have $\BQ_T = \bzero$ since $\BQ\BX = \BX\BQ = \bzero$ and thus $\BQ^{(a,b)}\bone_{n_b} = \bzero$ for all pairs of $(a,b).$ For $\BQ^{(a,a)}$ with $1\leq a\leq k$, $\BQ^{(a,a)}\bone_{n_a} = \bzero$ holds automatically by the definition of $\BQ$ in~\eqref{eq:Q}. For $a\neq b$, straightforward calculations lead to
\begin{equation}\label{eq:Q1}
\BQ^{(a,b)}\bone_{n_b} = -\frac{n_b z}{2}\left(\frac{1}{n_a} + \frac{1}{n_b}\right)\bone_{{n_a}} + \BM^{(a,b)}\bone_{n_b } - \BB^{(a,b)}\bone_{n_b} = \bzero.
\end{equation}
Thus, one has $ \frac{1}{n_b}\BB^{(a,b)}\BJ_{n_b\times n_b} =  \frac{1}{n_b}\BM^{(a,b)}\BJ_{n_b\times n_b }   -\frac{ z}{2}\left(\frac{1}{n_a} + \frac{1}{n_b}\right)\BJ_{{n_a}\times n_b}$ for all $a\neq b$, which implies $\BB_T^{(a,b)} = \BM_T^{(a,b)} - \frac{z(n_a + n_b)}{2n_an_b}\BJ_{n_a \times n_b}.$
The third formula in~\eqref{eq:eqv_opt} satisfies automatically.

\paragraph{$\eqref{eq:eqv_opt}\Longrightarrow\eqref{eq:optimal}$:} It suffices to prove $\BQ$ in \eqref{eq:Q} is positive semidefinite. 
By definition, the matrix $\BE^{(a,b)}$ is equal to $\frac{1}{2}\left( \frac{1}{n_a} + \frac{1}{n_b}\right) \BJ_{n_a\times n_b}$ and $\PP_{\TB}(\I_N) = \I_N - \BX$. Adding the first two formulas in~\eqref{eq:eqv_opt} blockwisely over all $(a, b)$ gives 
\begin{equation*}
z(\I_N-\BX) + \BM - \BB - z(\BE - \BX) = \underbrace{ z(\I_N - \BE) + \BM - \BB}_{\BQ} \succeq \bzero
\end{equation*}
where 
we have used the following facts:  $\BX^{(a,a)}=\BE^{(a,a)}$, $\BX^{(a,b)}=\bzero$ when $a\neq b$, $\BM^{(a,a)}_T=\bzero$ which follows from \eqref{def:M_special}, and $\BB^{(a,a)}_T=\bzero$ due to $\BB^{(a,a)}=\bzero$.
This shows $\BQ\succeq \bzero$. 
 \end{proof}

According to \eqref{eq:eqv_opt}, $\BB_T^{(a,b)}$ is determined by $\BM^{(a,b)}$ and $z$. So the only free variables are $z$ and $\BB_{\TB}^{(a,b)}$ for $a \neq b$.
To determine $z$, we replace $z\PP_{\TB} (\I_N) +\BM_{\TB} - \BB_{\TB} \succeq  \bzero$ by a stronger condition $z \geq \| \BM_{\TB} - \BB_{\TB}\|$ 
which clearly implies the former one. 
To choose $\BB_{\TB}^{(a,b)}$ for any $a \neq b$, notice that
\begin{equation*}
\BB^{(a,b)}> \bzero \Longleftrightarrow \BB_{\TB}^{(a,b)} + \BB_T^{(a,b)} > \bzero \Longleftrightarrow \BB_{\TB}^{(a,b)} > \frac{z(n_a + n_b)}{2n_an_b}\BJ_{n_a\times n_b} - \BM^{(a,b)}_T,
\end{equation*}
where 
we have used a substitution for $\BB_T^{(a,b)}$. 
To sum up, we have derived a replacement sufficient condition which guarantees  $\BX$ as the unique global minimum of \eqref{eq:prime}:
\begin{align}\label{cond:suff-2}
\begin{cases}
z \geq \| \BM_{\TB} - \BB_{\TB}\|, \\
\BB = \BB^\top, \\
\BB^{(a,a)} = \bzero, \quad \forall \, 1 \leq a \leq k, \\
\BB^{(a,b)}_T = \BM_T^{(a,b)} - \frac{z(n_a + n_b)}{2n_an_b}\BJ_{n_a\times n_b}, \quad \forall \, a \neq b, \\
\BB_{\TB}^{(a,b)} > \frac{z(n_a + n_b)}{2n_an_b}\BJ_{n_a\times n_b} - \BM^{(a,b)}_T, \quad \forall \, a \neq b. \\
\end{cases}
\end{align}

\subsection{Proof of Theorem~\ref{thm:main}} \label{sec:proof_main}
Now we are ready to prove the main theorem, which follows directly from the proposition below.
\begin{proposition}\label{prop:main}

Assume the proximity condition~\eqref{eq:prox1} holds for the partition $\{\Gamma_a\}_{a=1}^k$. We can choose $z$ and $\BB$ such that 
\begin{equation*}
z = \| \BM_{\TB} - \BB_{\TB} \|, \quad \BB_{\TB}^{(a,b)} = 4\bu_{a,b}\bu_{b,a}^{\top}, \quad \forall \, a \neq b,
\end{equation*}
and the sufficient condition in~\eqref{cond:suff-2} is satisfied.
Therefore, whenever the proximity condition holds, $\BX=\sum_{a=1}^k\frac{1}{|\Gamma_a|}\bone_{\Gamma_a}\bone_{\Gamma_a}^\top$ is the unique minimizer of the Peng-Wei relaxation of $k$-means.
\end{proposition}

\begin{lemma} \label{lem:D&MTB}
For any $1 \leq a,b \leq k$, $\BM_{\TB}^{(a,b)} = \BD_{\TB}^{(a,b)} = -2 \overline{\BX}_a  \overline{\BX}_b^\top$.
\end{lemma}

\begin{proof}
Let $\bx_{a,i}$ and $\bx_{b,j}$ be the $i$-th and $j$-th points in the $a$-th and $b$-th clusters, respectively. Then,
\begin{equation*}
\| \bx_{a,i} - \bx_{b,j} \|^2 = \| \bx_{a,i} \|^2 -2 \lag \bx_{a,i}, \bx_{b,j} \rag + \| \bx_{b,j} \|^2.
\end{equation*}
Denote  by $\bphi_a \in \RR^{n_a}$ and $\bphi_b\in\RR^{n_b}$  the column vectors consisted of $\| \bx_{a,i} \|^2$ and $\|\bx_{b,j}\|^2$, respectively. Then,
\begin{equation*}
\BD^{(a,b)} = \bphi_a \bone_{n_b}^\top - 2 {\BX}_a  {\BX}_b^\top + \bone_{n_a} \bphi_b^\top.
\end{equation*}
\begin{align*}
\BD_{\TB}^{(a,b)} &= (\I_{n_a} - \frac1 n_a \BJ_{n_a \times n_a}) \BD^{(a,b)} (\I_{n_b} - \frac1 n_b \BJ_{n_b \times n_b}) \\
    &=-2 (\I_{n_a} - \frac1 n_a \BJ_{n_a \times n_a}) {\BX}_a  {\BX}_b^\top (\I_{n_b} - \frac1 n_b \BJ_{n_b \times n_b}) \\
    &= -2 \overline{\BX}_a  \overline{\BX}_b^\top.
\end{align*}
The matrix $\BM$ is defined  in \eqref{def:M}, and  it is easy to check that $\BM_{\TB}^{(a,b)} = \BD_{\TB}^{(a,b)}$.
 \end{proof}

\begin{lemma} \label{lem:operator-bound}
The operator norm of $\BM_{\TB} - \BB_{\TB}$ is bounded by $2\sum_{l=1}^k \|\overline{\BX}_l\|^2$, i.e., 
\begin{equation*}
z= \| \BM_{\TB} - \BB_{\TB} \| \leq 2\sum_{l=1}^k\|\overline{\BX}_l\|^2.
\end{equation*}
\end{lemma}
\begin{proof}
Note that $\bu_{a,b} = \overline{\BX}_a \bw_{a,b}$ and by Lemma~\ref{lem:D&MTB}, $\BM_{\TB}^{(a,b)}= -2\overline{\BX}_a\overline{\BX}_b^{\top}$. Hence, $ \BB_{\TB} - \BM_{\TB} = 2 \widehat{\BX} \BW \widehat{\BX}^\top$, where $\widehat{\BX} \in \RR^{N \times mk}$ is defined as
\begin{equation*}
\widehat{\BX}^{(a,b)} = \bzero, \quad \widehat{\BX}^{(a,a)} = \overline{\BX}_a, \quad  \forall a \neq b,
\end{equation*}
and $\BW \in \RR^{mk \times mk}$ is given by
\begin{equation*}
\BW^{(a,b)} = \I_m -2\bw_{a,b}\bw_{a,b}^{\top}, \quad \BW^{(a,a)} = \I_m, \quad  \forall a \neq b.
\end{equation*}
Note that each $\BW^{(a,b)}$ is an orthogonal matrix and thus $\|\BW^{(a,b)}\| = 1.$ Let $\by$ be a vector of length $N$, and denote by $\by_a$ the $a$-th block of $\by$, $1\leq a\leq k$. There holds,
\begin{align*}
\left|\by^{\top}(\BM_{\TB} - \BB_{\TB})\by\right| 
& \leq 2 \sum_{a=1}^k\sum_{b=1}^k \left|\by^{\top}_a\overline{\BX}_a \BW^{(a,b)} \overline{\BX}_b\by_b^{\top}\right| \\
& \leq 2 \sum_{a=1}^k\sum_{b=1}^k \|\overline{\BX}_a\| \|\by_a\| \|\overline{\BX}_b\| \|\by_b\| \\
& \leq 2 \left( \sum_{l=1}^k \|\overline{\BX}_l\| \|\by_l\| \right)^2 \\
& \leq 2 \left(\sum_{l=1}^k \|\overline{\BX}_l\|^2\right) \left(\sum_{l=1}^k \|\by_l\|^2\right).
\end{align*}
Therefore, the operator norm of $\BM_{\TB} - \BB_{\TB}$ is bounded by $2\sum_{l=1}^k \|\overline{\BX}_l\|^2.$
 \end{proof}

 It only remains to check whether~\eqref{eq:prox} implies the second inequality in \eqref{cond:suff-2}:
\begin{equation} \label{cond:suff-3}
\BB_{\TB}^{(a,b)} = 4\bu_{a,b}\bu^{\top}_{b,a} > \frac{z(n_a + n_b)}{2n_an_b}\BJ_{n_a\times n_b} - \BM^{(a,b)}_T, \quad \forall \, a \neq b.
\end{equation}
To show this, we first derive an explicit expression for $\BM^{(a,b)}_{T}$.
\begin{lemma} \label{lem:D1}
For any $a \neq b$, there holds
\begin{equation*}
\frac{1}{n_b}\BD^{(a,b)} \bone_{n_b} - \frac{1}{n_a}\BD^{(a,a)}\bone_{n_a} = \left( h^2_{a,b} + \frac{1}{n_b}\|\overline{\BX}_b\|_F^2 - \frac{1}{n_a} \|\overline{\BX}_a\|_F^2 \right) \bone_{n_a} 
- 2h_{a,b} \bu_{a,b}.
\end{equation*}
\end{lemma}
\begin{proof}
The $i$-th entry of the left hand side is
\begin{align*}
(LHS)_i =&\frac{1}{n_b} \sum_{l=1}^{n_b} \|\bx_{a,i} - \bx_{b,l}\|^2 - \frac{1}{n_a}\sum_{l=1}^{n_a}\|\bx_{a,i} - \bx_{a,l}\|^2 \\
=& \|\bc_a- \bc_b\|^2 - 2 \lag \bx_{a,i} - \bc_a, \bc_b-\bc_a \rag  +\frac{1}{n_b}\sum_{l=1}^{n_b}  \|\bx_{b,l} - \bc_b\|^2 -\frac{1}{n_a}\sum_{l=1}^{n_a}\|\bx_{a,l} - \bc_a\|^2  \\
=& h^2_{a,b} - 2h_{a,b}(\overline{\BX}_a \bw_{b,a})_i + \frac{1}{n_b}\|\overline{\BX}_b\|_F^2 - \frac{1}{n_a} \|\overline{\BX}_a\|_F^2 \\
=&(RHS)_i.
\end{align*}
 \end{proof}
 
\begin{lemma}\label{lem:MT}
For any $a \neq b$, there holds
\begin{equation*}
\BM_T^{(a,b)} = h^2_{a,b}\BJ_{n_a\times n_b} - 2h_{a,b} \bu_{a,b}\bone_{n_b}^{\top} - 2h_{a,b} \bone_{n_a} \bu_{b,a}^{\top}.
\end{equation*}
\end{lemma}

\begin{proof}
By the definition of $\BM^{(a,b)}$  in \eqref{def:M}, 
\begin{align*}
\BM^{(a,b)}_T 
& = \BD^{(a,b)}_T  - \frac{1}{n_a}\BD^{(a,a)}\BJ_{n_a\times n_b} - \frac{1}{n_b} \BJ_{n_a\times n_b}\BD^{(b,b)}\\ 
& \quad + \frac{1}{2}\left( \frac{1}{n_a^2}\lag \BD^{(a,a)}, \BJ_{n_a\times n_a}\rag +  \frac{1}{n_b^2}\lag \BD^{(b,b)}, \BJ_{n_b\times n_b} \rag\right) \BJ_{n_a\times n_b} \\
& = \underbrace{ \frac{1}{n_b}\BD^{(a,b)}\BJ_{n_b\times n_b} - \frac{1}{n_a}\BD^{(a,a)}\BJ_{n_a\times n_b} }_{\Pi_1}  + \underbrace{ \frac{1}{n_a}\BJ_{n_a\times n_a}\BD^{(a,b)} - \frac{1}{n_b}\BJ_{n_a\times n_b}\BD^{(b,b)}}_{\Pi_2} \\
& \quad + \underbrace{ \left( \frac{1}{2n_a^2}\lag \BD^{(a,a)}, \BJ_{n_a\times n_a}\rag  +  \frac{1}{2n_b^2}\lag \BD^{(b,b)}, \BJ_{n_b\times n_b} \rag - \frac{1}{n_an_b}\lag\BD^{(a,b)}, \BJ_{n_a\times n_b}\rag \right) \BJ_{n_a\times n_a} }_{\Pi_3},
\end{align*}
where we have used 
\begin{equation*}
\BD^{(a,b)}_T = \frac{1}{n_a}\BJ_{n_a\times n_a}\BD^{(a,b)}  + \frac{1}{n_b}\BD^{(a,b)}\BJ_{n_b\times n_b} - \frac{1}{n_an_b}\lag \BD^{(a,b)}, \BJ_{n_a\times n_b}\rag\BJ_{n_a\times n_b}.
\end{equation*}
By Lemma~\ref{lem:D1}, we have
\begin{align*}
\Pi_1 &= \left( \frac{1}{n_b}\BD^{(a,b)} \bone_{n_b} - \frac{1}{n_a}\BD^{(a,a)}\bone_{n_a} \right) \bone_{n_b}^\top \\
    &= \left( h^2_{a,b} + \frac{1}{n_b}\|\overline{\BX}_b\|_F^2 - \frac{1}{n_a} \|\overline{\BX}_a\|_F^2 \right) \BJ_{n_a \times n_b} 
- 2h_{a,b} \bu_{a,b} \bone_{n_b}^\top.
\end{align*}
Similarly,
\begin{align*}
\Pi_2 &= \bone_{n_a} \left( \frac{1}{n_a}\BD^{(b,a)} \bone_{n_a} - \frac{1}{n_b}\BD^{(b,b)}\bone_{n_b} \right)^\top \\
     &=\left( h^2_{a,b} + \frac{1}{n_a}\|\overline{\BX}_a\|_F^2 - \frac{1}{n_b} \|\overline{\BX}_b\|_F^2 \right) \BJ_{n_a \times n_b}
- 2h_{a,b} \bone_{n_a} \bu_{b,a}^\top.
\end{align*}
Moreover, the $(i,j)$-entry of $\Pi_3$ is
\begin{align*}
 (\Pi_3)_{i,j}  
 & = \frac{1}{2n_a^2} \sum_{i=1}^{n_a}\sum_{j=1}^{n_a}  \|\bx_{a,i} - \bx_{a,j}\|^2 + \frac{1}{2n_b^2} \sum_{i=1}^{n_b}\sum_{j=1}^{n_b}  \|\bx_{b,i} - \bx_{b,j}\|^2  - \frac{1}{n_an_b}\sum_{i=1}^{n_a}\sum_{j=1}^{n_b} \|\bx_{a,i} - \bx_{b,j}\|^2 \\
& = \frac{1}{n_a} \sum_{i=1}^{n_a}\|\bx_{a,i} - \bc_a\|^2 + \frac{1}{n_b} \sum_{i=1}^{n_b}\|\bx_{b,i} - \bc_b\|^2 - \frac{1}{n_a} \sum_{i=1}^{n_a} \|\bx_{a,i} - \bc_a\|^2 \\
& \quad - \frac{1}{n_b}\sum_{j=1}^{n_b}\|\bx_{b,j} - \bc_b\|^2 - \|\bc_a - \bc_b\|^2 = - h_{a,b}^2.
\end{align*}
Adding up $(\Pi_1)_{i,j}$, $(\Pi_2)_{i,j}$ and  $(\Pi_3)_{i,j}$ leads to the desired identity.
 \end{proof}

\begin{proof}[\bf Proof of Proposition~\ref{prop:main}: ]
Combined with the explicit expression of $\BM^{(a,b)}_T$,~\eqref{cond:suff-3} is equivalent to
\begin{equation}\label{eq:positivity}
-4\bu_{a,b}\bu_{b,a}^{\top} + \left(\frac{z(n_a + n_b)}{2n_an_b} - h_{a,b}^2\right)\BJ_{n_a\times n_b} + 2h_{(a,b)} (\bu_{a,b}\bone_{n_b}^{\top} + \bone_{n_a}\bu_{b,a}^{\top}) < 0.
\end{equation}
By definition of $\tau_{a,b}$, we have
\begin{equation*}
\tau_{a,b} \geq \max(\bu_{a,b})  , \quad \tau_{a,b} \geq \max(\bu_{b,a}).
\end{equation*}
Define
\begin{equation*}
f(x,y) := -4xy  - 2h_{(a,b)} (x + y) + \frac{z(n_a + n_b)}{2n_an_b} - h_{(a,b)}^2.
\end{equation*}
Let $u_{a,b,i}$ and $u_{b,a,j}$ be the $i$-th and $j$-th entry of $\bu_{a,b}$ and $\bu_{b,a}$ respectively. One can easily see that $f(-u_{a,b,i}, -u_{b,a,j})$ is equal to the $(i,j)$-th entry of the matrix on the left hand side of~\eqref{eq:positivity}. 
Therefore, in order to prove~\eqref{eq:positivity}, it suffices to show that
$f(x,y) < 0$ for all $x,y \geq -\tau_{a,b}.$ Note that
if the proximity condition~\eqref{eq:prox} holds, then $2\tau_{a,b} \leq \|\bc_a - \bc_b\|$. Therefore, $x,y \geq -\tau_{a,b}\geq -\frac{1}{2}h_{a,b}.$

We claim that the maximum of $f(x,y)$ over $\{ (x,y)\in\RR^2:x \geq -\tau_{a,b}, y \geq -\tau_{a,b}\}$ is attained at $x = y = -\tau_{a,b}$ due to bilinearity of $f(x,y).$
More precisely, this follows from $2\tau_{a,b} \leq h_{a,b}$ and 
\begin{align*}
\frac{\pa f}{\pa x} & = -4y - 2h_{(a,b)} \leq 4\tau_{a,b} - 2h_{a,b} \leq 0, \\
\frac{\pa f}{\pa y} & = -4x - 2h_{(a,b)}  \leq 4\tau_{a,b} - 2h_{a,b} \leq 0
\end{align*}
over $\{ (x,y)\in\RR^2:x \geq -\tau_{a,b}, y \geq -\tau_{a,b}\}.$ 
Therefore,~\eqref{eq:positivity} holds if 
\begin{equation*}
\max_{\{x,y\geq -\tau_{a,b}\}} f(x,y) = -4\tau^2_{a,b} + 4h_{a,b}\tau_{a,b} - h_{a,b}^2 + \frac{z(n_a + n_b)}{2n_an_b} < 0.
\end{equation*}
Since $2\tau_{a,b} \leq h_{a,b}$, the inequality above is equivalent to
\begin{equation*}
h_{a,b} - 2\tau_{a,b} >  \sqrt{\frac{z(n_a + n_b)}{2n_an_b}}.
\end{equation*}
Meanwhile, the proximity condition implies
\begin{equation*}
h_{(a,b)} - 2\tau_{a,b} >\sqrt{\frac{\sum_{l=1}^k\|\overline{\BX}_l\|^2(n_a + n_b)}{n_an_b}} \geq \sqrt{\frac{z(n_a + n_b)}{2n_an_b}}.
\end{equation*}
Hence, we have $-4\tau^2_{a,b} + 4h_{a,b}\tau_{a,b} - h^2_{a,b} + \frac{z(n_a + n_b)}{2n_a n_b} < 0$ and~\eqref{eq:positivity} holds.
 \end{proof}

\subsection{Proof of Theorem~\ref{thm:lower}} \label{sec:lower}

This subsection is devoted to proving Theorem~\ref{thm:lower}, the necessary lower bound of $\frac{1}{2}h_{a,b} - \tau_{a,b}$ for $\BX=\sum_{a=1}^k\frac{1}{|\Gamma_a|}\bone_{\Gamma_a}\bone_{\Gamma_a}^\top$ to be a global minimum of the Peng-Wei relaxation of $k$-means.
We will use the necessary condition established in Proposition~\ref{thm:opt_no_unq} for the proof which states that, if $\BX$ is global minimizer, then there exist a number $z$ and a matrix $\BB$ obeying $\BB\geq 0$, $\BB^{(a,a)} = \bzero$ for all $1 \leq a \leq k$, and $\BQ=z (\I_N-\BE) + \BM - \BB\succeq 0$.
\begin{proof}[\bf Proof of Theorem~\ref{thm:lower}: ]
The proof is partitioned into three steps:
\paragraph{Step One:} We first show  that for any $a \neq b$, there holds
\begin{equation} \label{eq:B1}
h^2_{a,b}\bone_{n_a} - 2h_{a,b}\bu_{a,b} = \frac{z (n_a + n_b)}{2n_a n_b} \bone_{n_a} + \frac1 {n_b}\BB^{(a,b)} \bone_{n_b}.
\end{equation}
Note that $\lag \BD^{(a,a)}, \BJ_{n_a\times n_a}\rag = 2n_a\|\overline{\BX}_a\|_F^2$. By Lemma~\ref{lem:D1} and the definition of $\BM^{(a,b)}$  in \eqref{def:M}, we have
\begin{align*}
\BM^{(a,b)}\bone_{n_b} & = {n_b}\left(\frac{1}{n_b} \BD^{(a,b)}\bone_{n_b}  - \frac{1}{n_a}\BD^{(a,a)}\bone_{n_a}\right) \\
& \qquad + \frac{n_b}{2}\left( \frac{1}{n_a^2}\lag \BD^{(a,a)}, \BJ_{n_a \times n_a}\rag -  \frac{1}{n_b^2}\lag \BD^{(b,b)}, \BJ_{n_b \times n_b} \rag \right) \bone_{n_a} \\
& = n_b ( h^2_{a,b}\bone_{n_b} - 2h_{a,b}\bu_{a,b}) \\
& =  \frac{n_b z}{2}\left(\frac{1}{n_a} + \frac{1}{n_b}\right)\bone_{n_a} + \BB^{(a,b)}\bone_{n_b},
\end{align*}
where the last equation follows from~\eqref{eq:Q1}.

\paragraph{Step Two:} Next we establish a lower bound for $z$ and show that $z \geq 2\max\|\overline{\BX}_a\|^2$.
Combining $\BQ = z(\I_N - \BE) +\BM - \BB \succeq \bzero$ with $\BB^{(a,a)} = \bzero$ results in 
\begin{equation*}
\BQ^{(a,a)} = z\left(\I_{n_a} - \frac{1}{n_a}\BJ_{n_a\times n_a}\right) +\BM^{(a,a)} \succeq \bzero
\end{equation*} 
for all $1\leq a\leq k$. 
Also, Lemma~\ref{lem:D&MTB} and \eqref{def:M_special} imply  $\BM^{(a,a)} = \BM_{\TB}^{(a,a)} = -2\overline{\BX}_a\overline{\BX}_a^{\top}$. Therefore,   $z$ cannot be negative and
\begin{equation*}
z \I_{n_a} \succeq z \left(\I_{n_a} - \frac{1}{n_a}\BJ_{n_a\times n_a}\right)\succeq - \BM^{(a,a)} = 2 \overline{\BX}_a  \overline{\BX}_a^\top,
\end{equation*}
which gives $z \geq 2\max_{1\leq a\leq k}\|\overline{\BX}_a\|^2$.

\paragraph{Step Three:}
By applying $\BB \geq \bzero$ and $z\geq 2\max_{1\leq a\leq k}\|\overline{\BX}_a\|^2$ to~\eqref{eq:B1}, we get
\begin{equation*}
h^2_{a,b}\bone_{n_a} - 2h_{a,b}\bu_{a,b} 
\geq  \frac{z (n_a + n_b)}{2n_a n_b} \bone_{n_a} 
\geq \frac{\max\|\overline{\BX}_a\|^2(n_a + n_b)}{n_an_b}\bone_{n_a}.
\end{equation*}
Similarly, we have 
\begin{equation*}
h^2_{a,b}\bone_{n_b} - 2h_{a,b}\bu_{b,a} 
\geq \frac{\max\|\overline{\BX}_a\|^2(n_a + n_b)}{n_an_b}\bone_{n_b}.
\end{equation*}
Together they imply 
\begin{equation*}
h^2_{a,b} - 2h_{a,b}\tau_{a,b} \geq  \frac{\max\|\overline{\BX}_a\|^2(n_a + n_b)}{n_an_b},
\end{equation*}
where $\tau_{a,b} = \max\{ \max(\bu_{a,b}), \max(\bu_{b,a})\}$. 
 \end{proof}

\subsection{Proof of Proposition~\ref{prop:prox2}}\label{ss:prox2}
\begin{proof}[\bf Proof of Proposition~\ref{prop:prox2}: ]
It suffices to prove $\min_{1\leq i\leq n_a} \left\lag \bx_{a,i} - \frac{\bc_a + \bc_b}{2}, \bw_{b,a}\right\rag  = \frac{1}{2}h_{a,b} - \tau_{a,b}.$
For any $1 \leq i \leq n_a$, there holds
\begin{align*}
\left\lag \bx_{a,i} - \frac{\bc_a + \bc_b}{2}, \bw_{b,a} \right\rag 
& = \left\lag \bx_{a,i} - \bc_a + \frac{\bc_a - \bc_b}{2}, \bw_{b,a}\right\rag \\
& =  \lag \bx_{a,i} - \bc_a, \bw_{b,a} \rag + \frac{1}{2}\|\bc_a - \bc_b\| \\
& =  (\overline{\BX}_a \bw_{b,a} )_i + \frac{1}{2}\|\bc_a - \bc_b\| \\
& = -(\bu_{a,b})_i + \frac{1}{2}\|\bc_a - \bc_b\|.
\end{align*}
Similarly, for any $1 \leq j \leq n_b$, we have,
\begin{align*}
\left\lag \bx_{b,j} - \frac{\bc_a + \bc_b}{2}, \bw_{b,a} \right\rag = -(\bu_{b,a})_j + \frac{1}{2} \| \bc_a - \bc_b \|.
\end{align*}
Combining those two identities gives 
\begin{equation*}
\min_{a \neq b} \left\{ \frac{1}{2} h_{a,b} - \tau_{a,b} \right\} =  \min_{a \neq b} \min_{1 \leq i \leq n_a} \left\lag \bx_{a,i} - \frac{\bc_a + \bc_b}{2}, \bw_{b,a}\right\rag, 
\end{equation*}
which completes the proof.
 \end{proof}



\section{Proofs for Section~\ref{sec:balanced}} \label{sec:proof4balanced}

In this section, we provide concise proofs for Theorem~\ref{thm:main-bal} and Theorem~\ref{thm:lower-bal}. The proofs for the balanced case is parallel to the general case to a large extent. To avoid redundancy, we skip proofs and calculations that are basically the same as those in Section~\ref{sec:proof4det}. Also, we adopt similar notation as in Section~\ref{sec:proof4det} to emphasize the close relation between these two SDP relaxations of $k$-means.

\subsection{Proof of Theorem~\ref{thm:main-bal}}
Amini and Levina's relaxation is equivalent to the following optimization problem:
\begin{align}
\label{eq:prime_balanced}
\begin{split}
\min \quad &  \lag \BZ, \BD\rag \\
\mbox{s.t.} \quad &  \BZ \succeq 0, \quad \BZ \geq 0,  \quad \frac{1}{2} (\BZ +\BZ^\top) \bone_N = \bone_N,\quad \diag(\BZ) = \frac{1}{n}\bone_N.
\end{split}
\end{align}
In the standard form of a conic program, the optimization takes the form
\begin{equation}\label{eq:standard_balanced}
\min \quad \lag \BZ, \BD\rag, \quad \mbox{s.t.} \quad \A(\BZ) = \begin{bmatrix} \frac{1}{n} \bone_N \\ \bone_N  \end{bmatrix}, \quad \BZ \in {\cal K},
\end{equation}
where ${\cal K} = {\cal S}^{N}_+ \cap \RR^{N\times N}_{+}$ and the linear operator ${\cal A}$ is given by
\begin{equation*}
\A(\BZ) :  \quad \BZ \rightarrow \begin{bmatrix}
\diag (\BZ)  \\
\frac{1}{2}(\BZ+\BZ^\top)\bone_N
\end{bmatrix}.
\end{equation*}
Thus, it is effortless to derive the dual problem of Amini and Levina's relaxation using the duality theory of conic programming. The dual program reads
\begin{equation}\label{eq:dual_balanced}
\max \quad - \left\lag \frac{1}{n} \bz + \balpha, \bone_N\right\rag, \quad \mbox{s.t.} \quad \BD + \A^*\left( \blambda\right) \in {\cal K}^*,
\end{equation}
where $\blambda = \begin{bmatrix} \bz \\ \balpha \end{bmatrix}\in\RR^{2N}$ is the dual variable with respect to the affine constraints, ${\cal K}^* =  {\cal S}^{N}_+ + \RR^{N\times N}_{+}$ is the dual cone and
\begin{equation}\label{def:Aadj_balanced}
\A^*(\blambda) := \frac{1}{2}(\balpha\bone_N^{\top} + \bone_N \balpha^{\top}) + \diag (\bz)
\end{equation}
is the adjoint operator of $\A$ under the canonical inner product over $\RR^{N\times N}$, where $\diag(\bz)$ is the diagonal matrix whose diagonal is given by $\bz$.

We proceed to find the sufficient condition for $\BX = \sum_{a=1}^k \frac{1}{n} \bone_{\Gamma_a} \bone_{\Gamma_a}^\top$ to be the global minimum. Thanks to the conic duality theorem (Theorem~\ref{thm:conic}), we can prove the following lemma using the same construction as in Lemma~\ref{lem:slater}
\begin{lemma} \label{lem:slater_balanced}
$(\BX, \blambda)$ is a pair of primal/dual optima if and only  if the complementary slackness holds: $\lag \BD + \A^*(\blambda), \BX\rag = 0$ where $\BD +\A^*(\blambda)\in{\cal K}^*$.
\end{lemma}
\begin{proof}
It is easy to verify that $\widetilde{\BZ} = \frac{1-\lambda}{N} \bone_N \bone_N^\top + \lambda \I_N$ is strictly feasible for \eqref{eq:standard_balanced}, where $\lambda = \frac{k-1}{N-1} > 0$ for $k \geq 2$. As for the dual problem, we take $\balpha = \bzero$ and $\bz = z \bone_N$ where $z$ is a sufficiently large positive number, then$\BD + \A^*(\blambda) = \BJ_{N\times N} + \left( \BD + z\I_N - \BJ_{N\times N} \right)$ is inside the interior of ${\cal K}^*$.
\end{proof}

The task is to find $\bz$ and $\balpha$ such that the complementary slackness $\lag \BD + \A^*(\blambda), \BX\rag = 0$ is true. By definition, $\BD + \A^*\left( \blambda\right) = \BB + \BQ$, where $\BB \geq \bzero$ and $\BQ \succeq \bzero$. We choose $\bz$ such that 
\begin{equation*}
\bz_a = z_a \bone_n, \quad \forall 1 \leq a \leq k,
\end{equation*} 
where $z_1, \ldots, z_k$ are variables to be determined. In a similar fashion to Lemma~\ref{lem:comp_slack}, the complementary slackness gives
\begin{equation*}
\vct{\alpha}_a = -\frac{2}{n} \mtx{D}^{(a, a)} \vct{1}_n + \frac{1}{n^2} \langle \mtx{D}^{(a, a)}, \mtx{J}_{n\times n}\rangle \vct{1}_n - \frac{z_a}{n} \vct{1}_n.
\end{equation*}
As a result, matrix $\BB$ must satisfy
\begin{align*}
\BB^{(a,b)} > \bzero, \quad \BB^{(a,a)} = \bzero \quad \forall a\neq b.
\end{align*}
The matrix $\BQ$ is rewritten as
\begin{align} \label{eq:Q_balanced}
\BQ &= \BF + \BM - \BB,
\end{align}
where $\BM$ is defined the same as before:
\[
\mtx{M}^{(a, b)} = \mtx{D}^{(a, b)} - \frac{1}{n}\left[\mtx{D}^{(a, a)}\mtx{J}_{n\times n} + \mtx{J}_{n\times n} \mtx{D}^{(b, b)}\right] + \frac{1}{2n^2}\langle \mtx{D}^{(a, a)} + \mtx{D}^{(b, b)}, \mtx{J}_{n\times n} \rangle \mtx{J}_{n\times n}.
\]
and the matrix $\mtx{F}$ is given by:
\begin{align*}
\mtx{F}^{(a, b)} =  - \frac{z_a + z_b}{2n} \mtx{J}_{n\times n},  \quad \mtx{F}^{(a, a)} = z_a \left(\mtx{I}_n - \frac{1}{n} \mtx{J}_{n\times n} \right)  \quad  \forall a \neq b.
\end{align*}
Just the same as Proposition~\ref{thm:opt_no_unq}, the following optimality condition is not enough to guarantee that $\BX$ is a unique global minimum of \eqref{eq:prime_balanced}: $\BQ \succeq \bzero$ and $\BB \geq \bzero$ where $\BQ$ has the form of \eqref{eq:Q_balanced} and $\BB^{(a,a)} = \bzero$ for all $1 \leq a \leq k$. However, by following exactly the logic of the proof of Proposition~\ref{lem:optimal}, one can show its counterpart for the balanced case is still true:
\begin{proposition}[\bf A sufficient condition for the uniqueness of global minimum] \label{lem:optimal_balanced}
Any feasible pair of $\BQ \succeq \bzero$ and $\BB \geq \bzero$, where $\BQ$ has the form of \eqref{eq:Q_balanced}, $\BB^{(a,a)} = \bzero$ for all $1\leq a \leq k$, and $\BB^{(a,b)} > \bzero$ for all $a \neq b$, certifies $\BX$ to be a unique global minimum of \eqref{eq:prime}.
\end{proposition}
By following the argument of Proposition~\ref{prop:eqv_opt}, we can transform the condition for the uniqueness of global minimum into a more useful form.
\begin{proposition}
The optimality condition with uniqueness in Proposition~\ref{lem:optimal_balanced} is equivalent to
\begin{align} \label{cond:suff_balanced}
\begin{cases}
\mtx{F}_{T^\perp} + \mtx{M}_{T^\perp} - \mtx{B}_{T^\perp} \succeq \mtx{0},
\\
\mtx{M}_T^{(a, b)} - \mtx{B}_T^{(a, b)} -\frac{z_a + z_b}{2n} \mtx{J}_n  = \mtx{0}, \quad \forall a \neq b,
\\
\mtx{B}^{(a,b)} = (\mtx{B}^{(b,a)})^\top, \quad \mtx{B}^{(a,a)} = \mtx{0}, \quad \mtx{B}^{(a, b)} > \mtx{0}, \quad \forall a \neq b.
\end{cases}
\end{align}
\end{proposition}
Here, $T$ and $T^\perp$ are subspaces of $\RR^{N \times N}$ defined in Section~\ref{sec:suff_dual_cert}. The only free variables remained in \eqref{cond:suff_balanced} are $z_a$ and $\BB^{(a,b)}_{\TB}$. We choose them as
\begin{align} \label{eq:construct_balanced}
z_a = 2k \| \overline{\BX}_a \|^2, \quad \BB^{(a,b)}_{\TB} = 4 \bu_{a,b} \bu_{b,a}^\top, \quad \forall a \neq b.
\end{align}
Now we show that with such a construction leads to Theorem~\ref{thm:main-bal}.  In fact, Theorem~\ref{thm:main-bal} follows immediately from the proposition below as an implication of Proposition~\ref{lem:optimal_balanced}.
\begin{proposition}
Assume the proximity condition for balanced clusters \eqref{eq:prox-bal} holds for the partition $\{ \Gamma_a \}_{a=1}^k$. We can choose $z_a$ and $\BB$ such that both the sufficient condition \eqref{cond:suff_balanced} and \eqref{eq:construct_balanced} are satisfied.
\end{proposition}
\begin{proof}
It remains to prove $\BB^{(a,b)} > \bzero$ for all $a \neq b$ and $\BF_{\TB} + \BM_{\TB} - \BB_{\TB} \succeq \bzero$.
Notice that for all $a \neq b$
\[
\begin{cases}
\mtx{B}_T^{(a, b)} = - \frac{z_a + z_b}{2n} \mtx{J}_{n\times n} + \mtx{M}_T^{(a, b)},
\\
\mtx{B}_{T^\perp}^{(a, b)} = 4 \vct{u}_{a, b} \vct{u}_{b, a}^\top,
\end{cases}
\]
where $\BM^{(a,b)}_T$ is given by Lemma~\ref{lem:MT}. Then 
\[
\mtx{B}^{(a, b)} = 4 \vct{u}_{a, b} \vct{u}_{b, a}^\top  + \left( -\frac{z_a + z_b}{2n} + h_{a, b}^2 \right) \mtx{J}_{n\times n} - 2h_{a, b} ( \vct{u}_{a, b} \vct{1}_n^\top + \vct{1}_n \vct{u}_{b, a}^\top) .
\]
As with the proof of \eqref{eq:positivity} in Section~\ref{sec:proof_main}, it suffices to require
\[
h_{a, b} - 2 \tau_{a, b} > \sqrt{\frac{z_a + z_b}{2n}} = \sqrt{\frac{k}{n}\left(\|\overline{\mtx{X}}_a\|^2 + \|\overline{\mtx{X}}_b\|^2\right)},
\]
which is equivalent to the proximity condition for balanced clusters thanks to Proposition~\ref{prop:prox2}.

Next we show $\BF_{\TB} \succeq \BB_{\TB} - \BM_{\TB}$. Based on the proof of Lemma~\ref{lem:operator-bound}, we have $\BM_{\TB}^{(a,b)}= -2\overline{\BX}_a\overline{\BX}_b^{\top}$. Hence, $ \BB_{\TB} - \BM_{\TB} = 2 \widehat{\BX} \BW \widehat{\BX}^\top$, where $\widehat{\BX} \in \RR^{N \times mk}$ and $\BW \in \RR^{mk \times mk}$ are given by
\begin{align*}
\begin{cases}
\widehat{\BX}^{(a,b)} = \bzero, \quad \widehat{\BX}^{(a,a)} = \overline{\BX}_a, \quad & \forall a \neq b, \\
\BW^{(a,b)} = \I_m -2\bw_{a,b}\bw_{a,b}^{\top}, \quad \BW^{(a,a)} = \I_m, \quad & \forall a \neq b.
\end{cases}
\end{align*}
          Note that each $\BW^{(a,b)}$ is an orthogonal matrix and thus $\|\BW^{(a,b)}\| = 1.$ Let $\by \in \RR^N$ be a unit vector, and denote by $\by_a = \{ y_i \}_{i\in \Gamma_a}$, $1\leq a\leq k$. There holds,
\begin{equation*}
\by^{\top} \BW \by
\leq  \sum_{a=1}^k\sum_{b=1}^k \by^{\top}_a \BW^{(a,b)} \by_b^{\top} 
\leq \left( \sum_{l=1}^k \|\by_l\| \right)^2
\leq k \left( \sum_{l=1}^k \|\by_l\|^2 \right)
= k.
\end{equation*} 
This implies $\mtx{W} \preceq k \mtx{I}_{mk}$, which further implies
\begin{equation} \label{eq:oper_inq_balanced}
\mtx{B}_{\TB} - \mtx{M}_{\TB} 
\preceq 
2k \widehat{\BX} \widehat{\BX}^\top
\preceq 
\BG,
\end{equation}
where $\BG$ stands for
\begin{equation*}
\BG^{(a,b)} = \bzero, \quad \BG^{(a,a)} = z_a \I_n, \quad \forall a \neq b.
\end{equation*}
By the definition of $\BF$, it is easy to verify that $\BF_{\TB} = \BG_{\TB}$. Applying ${\cal P}_{\TB}$ to both sides of \eqref{eq:oper_inq_balanced} yields
\begin{equation*}
\BB_{\TB} - \BM_{\TB} \preceq \BF_{\TB}.
\end{equation*}
\end{proof}

\subsection{Proof of Theorem~\ref{thm:lower-bal}}
Lemma~\ref{lem:D&MTB} and \eqref{def:M_special} imply  $\BM^{(a,a)} = \BM_{\TB}^{(a,a)} = -2\overline{\BX}_a\overline{\BX}_a^{\top}$.
Since $\BQ \succeq \bzero$, $\BQ^{(a,a)} \succeq \bzero$ for any $a$. Using \eqref{eq:Q_balanced}, we have
\begin{equation*}
\BQ^{(a,a)} = \BF^{(a,a)} + \BM^{(a,a)} - \BB^{(a,a)} = z_a \left(\I_n - \frac{1}{n} \BJ_{n \times n} \right) - 2 \overline{\BX}_a \overline{\BX}_a^\top \succeq \bzero.
\end{equation*}
Thus,
\begin{equation*}
z_a \I_n \succeq  z_a \left(\I_n - \frac{1}{n} \BJ_{n \times n} \right) \succeq 2 \overline{\BX}_a \overline{\BX}_a^\top,
\end{equation*}
which gives $z_a \geq 2 \| \overline{\BX}_a \|^2$.
According to Lemma~\ref{lem:MT}, there holds
\begin{equation*}
\BM_T^{(a,b)} = h^2_{a,b}\BJ_{n \times n} - 2h_{a,b} \bu_{a,b}\bone_{n}^{\top} - 2h_{a,b} \bone_{n} \bu_{b,a}^{\top},
\end{equation*}
since for the balanced case $n_a = n$ for any $a$. Hence,
\begin{equation*}
\BM^{(a,b)} \bone_n = \BM_T^{(a,b)} \bone_n = n (h^2_{a,b} \bone_n - 2 h_{a,b} \bu_{a,b}).
\end{equation*}
On the other hand, by \eqref{cond:suff_balanced}, we have
\begin{equation*}
\BM^{(a,b)} \bone_n = \BM_T^{(a,b)} \bone_n = \BB^{(a,b)} \bone_n - \frac{z_a + z_b}{2} \bone_n.
\end{equation*}
Combining the above two equations with the fact that $\BB \geq \bzero$, we obtain the following estimation
\begin{equation*}
n (h^2_{a,b} \bone_n - 2 h_{a,b} \bu_{a,b}) = \BB^{(a,b)} \bone_n + \frac{z_a + z_b}{2} \bone_n \geq ( \| \overline{\BX}_a \|^2 +  \| \overline{\BX}_b \|^2) \bone_n.
\end{equation*}
This is equivalent to
\[h^2_{a,b} - 2 h_{a,b} \tau_{a,b} \geq \frac{ (\| \overline{\BX}_a \|^2 +  \|\overline{\BX}_b \|^2)}{n}.\]



\section{Proofs for Section~\ref{sec:sbm_and_gmm} }\label{sec:pf_rand}
In this section,  we apply the deterministic guarantee to  two typical random models  and prove Corollaries~\ref{cor:rbm} and~\ref{cor:gmm}. Each of the two models inherits a partition structure from how the data are sampled, which gives a ground truth of the underlying clusters.
We will discuss the sufficient condition for  the  exact recovery of the Peng-Wei relaxation based on the minimal separation between cluster centers.

\subsection{Key lemmas}
The main mathematical tools for the analysis are various concentration inequalities of random matrices as discussed in~\cite{Ver12} and~\cite{Tropp12}.
\begin{theorem}[\bf Matrix Bernstein inequality, Theorem 1.6 in~\cite{Tropp12}]\label{thm:bern}
Let $\{ \BZ_i \}_{i=1}^n$ be a sequence of real $d_1\times d_2$ random matrices. Assume that
\begin{equation*}
\mathbb{E} \BZ_i = 0, \quad \| \BZ_i \| \leq R, \quad \forall \, 1\leq i\leq n.
\end{equation*}
Consider the sum $\BS = \sum_{i=1}^n \BZ_i$, and denote
\begin{equation*}
\sigma^2(\BS) = \max \left\{ \left\| \sum_{i=1}^n \mathbb{E} [\BZ_i \BZ_i^{\top}] \right\|, \left\| \sum_{i=1}^n \mathbb{E} [\BZ_i^{\top} \BZ_i] \right\| \right\}.
\end{equation*}
Then for all $t \geq 0$,
\begin{equation*}
\Pr \left( \| \BS \| \geq t \right) \leq (d_1 + d_2) \cdot \exp \left( \frac{-t^2}{2\sigma^2(\BS) + 2Rt/3} \right).
\end{equation*}
\end{theorem}

\begin{lemma}[\bf Generalized stochastic ball model]\label{lem:rbm}
Let $\{ \ba_i \}_{i=1}^n$ be a sequence of i.i.d. random vectors in $\RR^m$ and assume each $\ba_i$ is a zero mean vector supported on the unit ball in $\RR^m$ with the  covariance matrix  given by $\bSigma$.
\begin{enumerate}
\item
Denote $\overline{\ba} = \frac{1}{n} \sum_{i=1}^n \ba_i$.  We have
\begin{equation}\label{eq:rbm1}
\Pr( \|\overline{\ba}\| \geq t) \leq (m+1) \cdot \exp\left(- \frac{nt^2}{2 + 2t/3} \right).
\end{equation}
\item 
Let $\BA$ be an $n \times m$ matrix whose $i$-th row is $\ba_i^{\top}$. Then 
\begin{equation}\label{eq:rbm2}
\Pr (\| \BA \| \geq   \sqrt{n(\|\bSigma\| + t)} ) \leq 2m\exp\left( -\frac{nt^2}{2 + 4t/3} \right).
\end{equation}
\end{enumerate}
\end{lemma}

\begin{proof}
Note that the distribution of each $\ba_i$ is supported on the unit ball with the covariance matrix given by $\bSigma$. Thus,
\begin{equation*}
\sigma^2\left(\sum_{i=1}^n\ba_i\right) =n\max\{ \| \bSigma\|, \Tr(\bSigma)\} \leq n,
\end{equation*}
which follows from $\|\E(\ba_i\ba_i^{\top})\| = \|\bSigma\|$ and $\|\E(\ba_i^{\top}\ba_i)\| = \Tr(\bSigma) \leq 1$. Moreover, there holds $\|\ba_i\| \leq 1$ and thus $R= \max_{1\leq i\leq n}\|\ba_i\|=1.$
Therefore, applying Theorem~\ref{thm:bern} immediately results in
\begin{align*}
\Pr(\|\overline{\ba}\| \geq t) & \leq (m+1) \cdot \exp\left(- \frac{nt^2}{2 + 2t/3} \right).
\end{align*}

For the second part, first note that $\|\BA\|^2 = \|\BA^{\top}\BA\| = \left\| \sum_{i=1}^n \ba_i\ba_i^{\top}\right\|$. Let $\BZ_i = \ba_i\ba_i^{\top} - \bSigma$ be a centered random matrix and its operator norm is controlled by
\begin{equation*}
R = \max_{1\leq i\leq n}\|\BZ_i\| \leq \max_{1\leq i\leq n}\|\ba_i\|^2 + \|\bSigma\| \leq 2.
\end{equation*}
For the variance of $\BZ_i$, since $\E(\BZ_i\BZ_i^{\top}) = \E(\BZ_i^{\top}\BZ_i) = \E(\|\ba_i\|^2\ba_i\ba_i^{\top}) - \bSigma^2$, we have $-\bSigma^2 \preceq \E(\BZ_i\BZ_i^{\top}) \preceq \bSigma$. Therefore,
\begin{equation*}
\|\E(\BZ_i\BZ_i^{\top})\| \leq \max\{\|\bSigma\|^2, \|\bSigma\|\} = \|\bSigma\| \leq 1
\end{equation*}
and $\sigma^2(\sum_{i=1}^n \BZ_i) \leq n.$ 
Applying Theorem~\ref{thm:bern} again gives
\begin{align*}
\Pr\left( \left\| \sum_{i=1}^n \BZ_i \right\| \geq n t \right) & \leq 2m \cdot \exp\left(-\frac{n^2t^2}{ 2\sigma^2(\BS)+ 2Rnt/3} \right) \\
& \leq 2m \cdot \exp\left( -\frac{ nt^2}{2 + 4t/3} \right).
\end{align*} 
Therefore, since $\|\BA\|^2 \leq \|\sum_{i=1}^n \BZ_i\| + n\|\bSigma\|$, we have
\begin{equation*}
\|\BA\| \leq \sqrt{n(\|\bSigma\| + t)}  
\end{equation*}
with probability at least $1 -2m\exp\left( -\frac{nt^2}{2 + 4t/3} \right).$

 \end{proof}

\begin{lemma}[\bf Gaussian mixture model]\label{lem:gmm}
Let $\{ \ba_i \}_{i=1}^n$ be a sequence of i.i.d. random vectors in $\RR^m$ sampled from multivariate Gaussian distribution $\mathcal{N}(\boldsymbol{0},\bSigma)$.
\begin{enumerate}
\item Denote $\overline{\ba} = \frac{1}{n} \sum_{i=1}^n \ba_i$. There holds
\begin{equation}\label{eq:gmm-1}
\Pr\left( \|\overline{\ba}\| \geq \sqrt{\frac{m(1+t)\|\bSigma\|}{n}}\right) \leq \max \{e^{-mt/8}, e^{-mt^2/8} \}, \quad \forall t\geq 0.
\end{equation}
\item Let $\BA$ be $n \times m$ matrix whose $i$-th row is $\ba_i^{\top}$, then for any $t \geq 0$
\begin{equation}\label{eq:gmm-2}
\Pr (\| \BA \| \geq \sqrt{ \|\bSigma\|}(\sqrt{n} +\sqrt{m}+t)) \leq 2 e^{-t^2/2}.
\end{equation}
\item Let $\sigma_{\min}$ be the smallest singular value of $\bSigma$, then for any $t \geq 0$
\begin{equation}
\Pr ( (\| \BA \| \leq \sigma_{\min} (\sqrt{n} - \sqrt{m} - t) ) \leq 2 e^{-t^2/2},
\end{equation}
\end{enumerate}
\end{lemma}

\begin{proof}
Obviously, the sample mean $\overline{\ba}$ is a random vector satisfying $\mathcal{N}(\boldsymbol{0},\frac{1}{n} \bSigma)$. Due to the rotational invariance, it can be rewritten as
$\overline{\ba} = \frac{1}{\sqrt{n}}\bSigma^{1/2}\bw$
where $\bw\sim\mathcal{N}(\bzero, \I_m).$
Note that $\|\bw\|^2$ is a $\chi^2_m$ random variable with $\E(\|\bw\|^2) = m$ and 
\begin{equation*}
\Pr(\|\bw\|^2 - m \geq t) \leq \exp\left(-\frac{t^2}{8m}\right) \vee \exp\left(-\frac{t}{8}\right).
\end{equation*}
It is easy to see that $\|\overline{\ba}\| \leq \sqrt{\frac{m(1+t)\|\bSigma\|}{n}}$ holds with probability at least $1 - \max \{e^{-mt/8}, e^{-mt^2/8} \}.$

For the second and the third part, we use similar techniques by first rewriting $\BA$ as
$\BA = \BW \bSigma^{1/2}$ where $\BW$ is an $n\times m$ standard Gaussian random matrix. Corollary 5.35 in~\cite{Ver12} implies that $\sqrt{n} - \sqrt{m} -t \leq \|\BW\| \leq \sqrt{n} + \sqrt{m} + t$ holds with probability at least $1 - e^{-t^2/2}$. Therefore, 
\begin{equation*}
\sigma_{\min} (\sqrt{n} - \sqrt{m} - t) \leq \|\BA\| \leq \sqrt{\|\bSigma\|}(\sqrt{n} + \sqrt{m} + t)
\end{equation*}
holds with probability at least $1 - 2e^{-t^2/2}.$
 \end{proof}

\begin{lemma}\label{lem:gauss}
For two independent standard Gaussian random vectors $\bx$ and $\by$ in $\RR^m$, there holds
\begin{equation}\label{eq:gmm1}
\Pr(\bx^{\top}\bmu \geq t\|\bmu\|) \leq e^{-t^2/2}, \quad \forall t \geq 0,
\end{equation}
for a fixed deterministic vector $\bmu$. Also, we have
\begin{equation}\label{eq:gmm2}
\Pr( \bx^{\top}\bPsi \by \geq m\sqrt{t(1+t)}\|\bPsi\|) \leq 2 \max \{e^{-mt/8}, e^{-mt^2/8} \} , \quad \forall t \geq 0,
\end{equation}
for a fixed matrix $\bPsi$ and $t\geq 1$. Moreover,
\begin{equation}\label{eq:gmm3}
\Pr(\bx^{\top}\bSigma\bx - \Tr(\bSigma) \geq t) \leq \exp\left(-\frac{t^2}{8\|\bSigma\|_F^2}\right)\vee \exp\left(-\frac{t}{8\|\bSigma\|}\right), \quad \forall t \geq 0,
\end{equation}
for a fixed positive semidefinite matrix $\bSigma$.

\end{lemma}

\begin{proof}

Note that $\bx^{\top}\bmu/\|\bmu\|$ is a standard Gaussian random variable. For a standard Gaussian random variable $g$, we have $\Pr(g\geq t)\leq \frac{1}{2}e^{-t^2/2}$, which can be easily verified as follows:
\begin{align*}
\Pr(g\geq t) &= \frac{1}{\sqrt{2\pi}}\int_t^\infty e^{-x^2/2}dx\\
&=e^{-t^2/2}\frac{1}{\sqrt{2\pi}}\int_t^\infty e^{-\frac{(x+t)(x-t)}{2}}dx\\
&\leq e^{-t^2/2}\frac{1}{\sqrt{2\pi}}\int_t^\infty e^{-\frac{(x-t)^2}{2}}dx=\frac{1}{2}e^{-t^2/2}.
\end{align*}

For~\eqref{eq:gmm2}, first note that $\| \by \|^2$ is a chi-squared variable with $m$ degree of freedom, hence
\begin{equation*}
\|\bPsi \by\|\leq \|\bPsi\| \|\by\| \leq \sqrt{m(1+t)} \|\bPsi\|
\end{equation*}
holds with probability at least $1 - \max \{e^{-mt/8}, e^{-mt^2/8} \}$. Conditioned on the event $\{\| \bPsi \by\|\leq \sqrt{m(1+t)} \| \bPsi \| \}$, 
$\bx^{\top}\bPsi\by$ is a Gaussian random variable with variance at most $m(1+t)\|\bPsi\|^2$. As a result,
\begin{equation*}
\Pr(\bx^{\top}\bPsi\by \geq m\sqrt{t(1+t)}\|\bPsi\|) \leq e^{-mt/2}
\end{equation*}
and $\bx^{\top}\bPsi\by\geq m\sqrt{t(1+t)}$ holds with probability at least $1 - 2 \max \{e^{-mt/8}, e^{-mt^2/8} \}.$

For~\eqref{eq:gmm3}, we use the rotational invariance as well as the eigen-decomposition of $\bSigma$, i.e., $\bSigma = \BU^{\top}\diag(\lambda_1, \cdots, \lambda_m)\BU$ with $\lambda_i \geq 0$ for $1\leq i\leq m$. Therefore, $\bx^{\top}\bSigma\bx$ is the sum of weighted $\chi^2_1$ random variables where
\begin{equation*}
\bx^{\top}\bSigma\bx = \sum_{i=1}^m \lambda_i \xi^2_i, \quad \xi_i = (\BU \bx)_i, \quad \E(\bx^{\top}\bSigma\bx) = \Tr(\bSigma).
\end{equation*}
After applying Bernstein inequality, we get the desired result where $\max_{i}\lambda_i = \|\bSigma\|$ and $\sum_{i=1}^m \lambda_i^2 = \|\bSigma\|^2_F$.
 \end{proof}

\subsection{Stochastic ball model} \label{sec:sbm}
In this subsection, we  prove Corollary~\ref{cor:rbm} for the generalized stochastic ball model. It extends the results in~\cite{iguchi2015tightness,IguchiMPV17,AwasBCKVW15} where the probability distributions are assumed to the same and isotropic for all the clusters. The question is how large the minimal separation $\Delta = \min_{a \neq b} \| \bmu_a - \bmu_b \|$ should be in order to to ensure the exact recovery of the Peng-Wei relaxation with high probability. An outline of the proof of Corollary~\ref{cor:lower-rbm} is also given at the end of the subsection.

\begin{proof}[\bf Proof of Corollary~\ref{cor:rbm}]
It suffices to estimate $\|\overline{\BX}_a\|$, $h_{a,b}$ and $\tau_{a,b}$ for all $a\neq b$. We will bound those quantities on the premise that~\eqref{eq:rbm1} and~\eqref{eq:rbm2}, i.e., 
\begin{equation}\label{eq:rbm-cond}
\|\BX_a - \bone_{n_a}\bmu_a^{\top}\| \leq \sqrt{n_a(\|\bSigma_a\| + t)} \quad\mbox{and}\quad \|\bc_a - \bmu_a\|  \leq t,
\end{equation}
hold  for all $1\leq a\leq k$ with probability for all $1\leq a\leq k$, at least $1 - 4km\exp(-\frac{Nw_{\min} t^2}{2 + 4t/3}).$
\paragraph{Estimation of $\|\overline{\BX}_a\|$:}
By the triangle inequality, the operator norm of $\overline{\BX}_a$ can be bounded  from above as 
\begin{align*}
\|\overline{\BX}_a\| 
& = \|\BX_a - \bone_{n_a}\bc_a^{\top}\| \\
& \leq \|\BX_a - \bone_{n_a} \bmu_a^{\top}\| + \sqrt{n_a} \|\bc_a - \bmu_a\| \\
& \leq \sqrt{n_a(\| \bSigma_a \| + t)} + t \sqrt{n_a}  
\end{align*}
for all $1\leq a\leq k$ with probability at least $1 - 4km\exp\left( -\frac{Nw_{\min}t^2}{2 + 4t/3} \right).$

\paragraph{Estimation of $\tau_{a,b}$ and $h_{a,b}$:}
Recall that
$\tau_{a,b} = \max\{\max\{ \overline{\BX}_a\bw_{a,b}\}, \max\{\overline{\BX}_b \bw_{b,a}\} \}.$ For each entry of $\overline{\BX}_a\bw_{a,b}$, we have 
\begin{equation*}
(\overline{\BX}_a\bw_{a,b})_i \leq \| \bx_{a,i} - \bmu_{a}\| + \| \bc_a -  \bmu_{a} \| \leq 1 + t
\end{equation*}
which follows from $\|\bx_{a,i} - \bmu_a\| \leq 1$ and~\eqref{eq:rbm-cond}.
A similar bound holds for $\overline{\BX}_b\bw_{b,a}$ and thus under the event where \eqref{eq:rbm-cond} holds, $\tau_{a,b} \leq 1 + t$ holds for all $a\neq b$ with probability at least $1 - 4km\exp({-\frac{Nw_{\min}t^2}{2 + 4t/3}})$.

For $h_{a,b}$, it has a simple lower bound: 
\begin{equation*}
h_{a,b} = \| \bc_a - \bc_b \| \geq \| \bmu_a - \bmu_b \| - \| \bc_a - \bmu_a\| - \| \bc_b - \bmu_b\| \geq \Delta - 2t.
\end{equation*}
Therefore, a lower bound of $\frac{1}{2}h_{a,b} - \tau_{a,b}$ is
\begin{equation*}
\frac{1}{2}h_{a,b} - \tau_{a,b} \geq \frac{1}{2}\Delta - t - (1 + t) = \frac{1}{2}\Delta - 2t - 1,
\end{equation*}
which holds uniformly over all $(a,b)$ with probability at least $1 - 4km\exp({-\frac{Nw_{\min}t^2}{2 + 4t/3}}).$
\paragraph{Proximity condition for stochastic ball model:}
Now we wrap up our discussion and apply the proximity condition~\eqref{eq:prox1}. For each $a$, it follows from $\| \overline{\BX}_a\| \leq ( \sqrt{\|\bSigma_a\| + t} + t) \sqrt{n_a}$ that
\begin{align*}
\sum_{a=1}^k \|\overline{\BX}_l\|^2 
& \leq \sum_{a=1}^k ( \|\bSigma_a\| + t + 2t\sqrt{\|\bSigma_a\| + t} + t^2 ) n_a \\
& \leq ( \sigma^2_{\max} + t + 2t(\sigma_{\max} + \sqrt{t}) + t^2 ) N \\
& \leq \left[ (\sigma_{\max} + t)^2 + t + 2t^{3/2}\right] N,
\end{align*}
where the second line follows from $\|\bSigma_a\|\leq \sigma^2_{\max}$ and $\sqrt{\|\bSigma_a\| + t} \leq \sqrt{\|\bSigma_a\|} + \sqrt{t}.$

Therefore, for all pairs of $a$ and $b$, the proximity condition~\eqref{eq:prox1} for the generalized stochastic ball model is guaranteed if
\begin{equation}\label{eq:rbm-prox}
\Delta \geq 2 + 4t  + \sqrt{  \frac{ 2\left( (\sigma_{\max} + t)^2 + t + 2t^{3/2}\right) }{w_{\min}} },
\end{equation}
which holds with probability at least $1 - 4km \exp({-\frac{Nw_{\min}t^2}{2+4t/3}})$.
Now we choose $t = \sqrt{\frac{4\log(4kmN^{\gamma}) }{Nw_{\min}}}$. We further assume that $N \geq \frac{4}{w_{\min}} \log(4 kmN^{\gamma})$, then $t \leq 1$ and ~\eqref{eq:rbm-prox} holds with probability at least 
\begin{equation*}
1 - 4km \exp\left(-\frac{Nw_{\min} \cdot t^2}{2 + 4t/3}\right) \geq 1 - 4km \exp\left(-\frac{1}{4} Nw_{\min} \cdot t^2 \right) \geq 1 - N^{-\gamma}.
\end{equation*}
Note that $w_{\min} \leq \frac{1}{k} \leq \frac{1}{2}$ and $t \leq 1$. By enlarging the right hand side of \eqref{eq:rbm-prox} as the following,
\begin{align*}
2 + 4t  + \sqrt{  \frac{ 2\left( (\sigma_{\max} + t)^2 + t + 2t^{3/2}\right) }{w_{\min}} }
&\leq 2 + \sqrt{\frac{2}{w_{\min}}} \sigma_{\max} + \sqrt{\frac{t}{w_{\min}}} + (4 + \sqrt{\frac{2}{w_{\min}}}) t +  \sqrt{\frac{2 t^{3/2}}{w_{\min}}}  \\
& \leq 2 + \sqrt{\frac{2}{w_{\min}}} \sigma_{\max} + 7 \sqrt{\frac{t}{w_{\min}}},
\end{align*}
we derive a sufficient condition of \eqref{eq:rbm-prox} which guarantees the proximity condition \eqref{eq:prox1} for the stochastic ball models with probability at least $1 - N^{-\gamma}$:
\begin{equation*}
\Delta \geq 2 + \sqrt{\frac{2}{w_{\min}}} \sigma_{\max} + 7 \sqrt{\frac{t}{w_{\min}}}.
\end{equation*}

In particular, if $n_a = n$ for all $a$ and each $\mathcal{D}_a$ is the uniform distribution over $\RR^m$, there holds $\sigma^2_{\max} = \|\bSigma_a\| = \frac{1}{m+2}$ and \eqref{eq:rbm-prox} can be simplified into
\begin{equation*}
\Delta \geq 2 + \sqrt{\frac{2k}{m+2}} + 7\sqrt{tk}
\end{equation*}
which completes the proof.
 \end{proof}

The necessary lower bound (Theorem~\ref{thm:lower}) can also be applied to the generalized stochastic ball model. For the sake of simplicity, we restrict our discussion to the special case where distributions are all uniform distributions over the unit balls and clusters are balanced, i.e., $n_a = n, \, \forall 1 \leq a \leq k$.
\begin{proof}[\bf Proof outline of Corollary~\ref{cor:lower-rbm}: ]
For each pair of $a$ and $b$, $\tau_{a,b} > 1 - \epsilon$ with high probability for any $\epsilon > 0$, provided that $N$ is large. As for the operator norms, Theorem 5.41 in~\cite{Ver12} implies that $\| \overline{\BX}_a \| \geq (1-\epsilon) \sqrt{\frac{n}{m+2}}$ with high probability. Simple calculations show that the necessary lower bound \eqref{eq:necessary} is equivalent to
\begin{equation} \label{eq:nec_eqv}
h_{a,b} \geq \tau_{a,b} + \sqrt{\tau_{a,b}^2 + \frac{2}{n} \max{ \| \overline{\BX}_a \|^2} }, \quad \forall a \neq b.
\end{equation}
Adding up all these together, we yield the necessary lower bound for the special case as in Corollary~\ref{cor:lower-rbm}.
 \end{proof}

\subsection{Gaussian mixture model}\label{sec:gmm}

In this subsection, we  prove Corollary~\ref{cor:gmm} for the Gaussian mixture model. We  still focus on the minimal separation condition for the exactness of the Peng-Wei relaxation. Denote $p(t) = \max \{e^{-mt/8}, e^{-mt^2/8} \}$.
\begin{proof}[\bf Proof of Corollary~\ref{cor:gmm}: ]
Let $N$ be the number of points drawn from the Gaussian mixture model and $n_a$ be the number of points belonging to $\mathcal{N}(\bmu_a, \bSigma_a)$. To simplify our analysis, we assume $n_a = w_a N$ and $\bx_{a,i}\sim \mathcal{N}(\bmu_a, \bSigma_a)$ for all $1\leq a\leq k.$

\paragraph{Estimation of $\|\overline{\BX}_a\|$:}
Let $\BX_a \in \RR^{n_a \times m}$ be the data drawn from $\mathcal{N}(\bmu_a, \bSigma_a)$. Lemma~\ref{lem:gmm} states that the sample mean $\bc_a = \frac{1}{n_a}\sum_{i=1}^{n_a}\bx_{a,i}$ satisfies $\|\bc_a - \bmu_a\| \leq \sqrt{\frac{m(1+t)\|\bSigma_a\|}{n_a}}$ for all $a$ with probability at least $1 - k \cdot p(t)$. 
Considering $\|\overline{\BX}_a\|$, it obeys
\begin{align*}
\|\overline{\BX}_a\| & \leq \|\BX_a - \bone_{n_a}\bmu_a^{\top}\| + \sqrt{n_a}\|\bc_a - \bmu_a\| \\
& \leq \sqrt{\|\bSigma_a\|}(\sqrt{n_a} + \sqrt{m} + \sqrt{mt} + \sqrt{m(1+t)}) \\
& \leq \sqrt{\|\bSigma_a\|} (\sqrt{n_a} + 2\sqrt{m}(1 + \sqrt{t}))
\end{align*}
for all $1\leq a\leq k$
with probability at least $1 - 2ke^{-mt/2}$,  where we have used~\eqref{eq:gmm-2} in the second line. It follows that
\begin{align*}
\frac{\sum_{l=1}^k \| \overline{\BX_l} \|^2 (n_a + n_b)}{4 n_a n_b} 
&\leq \frac{1}{2N} \left( \sum_{l=1}^k \| \bSigma_l \| \left(n_a + 8m(1+t)\right) \right) \left( \frac{1}{w_a} + \frac{1}{w_b} \right) \\
& \leq \frac{\sigma^2_{\max}}{Nw_{\min}}  \left(  N + 8km(1 + t) \right) \\
& \leq \frac{\sigma^2_{\max}}{w_{\min}} \left( 1 + \frac{8km(1+t)}{N}\right),
\end{align*}
where $w_{\min} = \frac{1}{N}\min_{1\leq l\leq k} n_l$ and $w_{\min} \leq \frac{1}{k}.$
Therefore, for all $a\neq b$ and all $t \geq 0$, the right hand side of \eqref{eq:prox1} is bounded from above by
\begin{align}
\sqrt{\frac{\sum_{l=1}^k \| \overline{\BX_l} \|^2 (n_a + n_b)}{4 n_a n_b}} 
& \leq 
\sqrt{\frac{\sigma^2_{\max}}{w_{\min}} \left( 1 + \frac{8km(1+t)}{N}\right)} \nonumber \\
&\leq \frac{\sigma_{\max}}{\sqrt{w_{\min}}} \left( 1 + \sqrt{\frac{8km(1+t)}{N}}\right)   \label{eq:gmm-upbd}
\end{align}
with probability at least $1 - k \cdot p(t) -2ke^{-mt/2}$, which is greater than $1 - 3k \cdot p(t)$.

\paragraph{Estimation of $\tau_{a,b}$ and $h_{a,b}$: }
For $h_{a,b}$, it follows from Lemma~\ref{lem:gmm} that
\begin{align}
h_{a,b} & = \|\bc_a - \bc_b\| \geq \|\bmu_a - \bmu_b\| - \|\bc_a - \bmu_a\| - \|\bc_b - \bmu_b\| \nonumber \\
& \geq \|\bmu_a - \bmu_b\|- \sqrt{m(1+t)\sigma^2_{\max}}\left(\frac{1}{\sqrt{n_a}} + \frac{1}{\sqrt{n_b}}\right) \nonumber \\
& \geq \|\bmu_a - \bmu_b\|- 2\sigma_{\max} \sqrt{\frac{ m(1+t)}{N w_{\min}}} \label{eq:gmm-h1}
\end{align}
holds with probability at least $1 - 2ke^{-mt/8}$ for any $a$ and $b$. Further assume $N \geq \frac{16 \sigma_{\max}^2 m(1+t)}{\Delta^2 w_{\min}}$, then
\begin{equation} \label{eq:gmm-h2}
h_{a,b} \geq \frac{\|\bmu_a - \bmu_b\|}{2}.
\end{equation}

Note that $\bu_{a,b}$ is defined as $\bu_{a,b} = \overline{\BX}_a \bw_{a,b}$ and each entry of $\bu_{a,b}$ is given by
$(\bu_{a,b})_i = \frac{1}{h_{a,b}} (\bx_{a,i} - \bc_a)^{\top} (\bc_a - \bc_b)$. To get an upper bound for $\bu_{a,b}$, it suffices to bound $(\bx_{a,i} - \bc_a)^{\top}(\bc_a - \bc_b)$, which can be partitioned into three terms:
\begin{equation*}
(\bx_{a,i} - \bc_a)^{\top} (\bc_a - \bc_b)  = \underbrace{ (\bx_{a,i} - \bmu_a )^{\top} (\bc_a - \bmu_a) }_{J_1}  + \underbrace{(\bx_{a,i} - \bc_a)^{\top}(\bmu_a - \bc_b)}_{J_2} - \| \bc_a - \bmu_a\|^2.
\end{equation*}
\begin{enumerate}[1.]
\item For $J_1$, note that $\bx_{a,i} - \bmu_a$ and $\bc_a - \bmu_a$ are not completely independent from each other. Thus we further decompose $J_1$ into
\begin{equation*}
(\bx_{a,i} - \bmu_a)^{\top} (\bc_a - \bmu_a) = \frac{1}{n_a}\|\bx_{a,i} - \bmu_a\|^2 + \frac{1}{n_a}(\bx_{a,i} - \bmu_a)^{\top} \left(\sum_{j\neq i} (\bx_{a,j} - \bmu_a)\right).
\end{equation*}
For the first term above,~\eqref{eq:gmm3} implies $\|\bx_{a,i} - \bmu_a\|^2\leq m(1+t)\|\bSigma_a\|$ with probability at least $1 - e^{-mt/8}.$
For the second term, we can reformulate it as 
\begin{equation*}
 \frac{1}{n_a}(\bx_{a,i} - \bmu_a)^{\top} \left(\sum_{j\neq i} (\bx_{a,j} - \bmu_a)\right) = \left\lag\bw, \frac{1}{n_a}\bSigma_a^{1/2} \sum_{j\neq i}(\bx_{a,j} - \bmu_a)  \right\rag
\end{equation*}
where $\bw\sim \mathcal{N}(\bzero, \I_m)$ and $\bw$ is independent of $\frac{1}{n_a} \sum_{j\neq i}(\bx_{a,j} - \bmu_a) \sim \mathcal{N}\left( \bzero, \frac{n_a - 1}{n_a^2}\bSigma_a \right)$. Applying~\eqref{eq:gmm2} implies
\begin{equation*}
(\bx_{a,i} - \bmu_a)^{\top} \left(\sum_{j\neq i} (\bx_{a,j} - \bmu_a)\right) \leq m\|\bSigma_a\|  \sqrt{\frac{t(1+t)}{n_a}}
\end{equation*}
with probability at least $1 - 2 \cdot p(t)$. So we can conclude that 
\begin{equation*} 
J_1 \leq m\|\bSigma_a\|\left( \frac{1 +t }{n_a} +   \sqrt{\frac{t(1+t)}{n_a}}\right) 
\end{equation*}
for all $a$ with probability at least $1 - 3 N \cdot p(t)$, for all $t \geq 0$.

\item For $J_2$, we decompose it into two terms:
\begin{equation*}
(\bx_{a,i} - \bc_a)^{\top}(\bmu_a - \bc_b) = (\bx_{a,i} - \bc_a)^{\top}(\bmu_a - \bmu_b)  + (\bx_{a,i} - \bc_a)^{\top}(\bmu_b - \bc_b).
\end{equation*}
Since $(\bx_{a,i} - \bc_a)^{\top}(\bmu_a - \bmu_b)\sim \mathcal{N}(0, \frac{n_a-1}{n_a}(\bmu_a- \bmu_b)^{\top}\bSigma_a(\bmu_a - \bmu_b))$,~\eqref{eq:gmm1} indicates 
\begin{equation*}
(\bx_{a,i} - \bc_a)^{\top}(\bmu_a - \bmu_b) \leq \sqrt{s (\bmu_a- \bmu_b)^{\top}\bSigma_a(\bmu_a - \bmu_b)}
\end{equation*}
for all $(a,b,i)$ with probability at least $1 - kNe^{-s/2}.$
On the other hand,~\eqref{eq:gmm2} directly gives 
\begin{equation*}
(\bx_{a,i} - \bc_a)^{\top}(\bmu_b - \bc_b) \leq m\sqrt{\frac{ t(1+t)\|\bSigma_a\|\|\bSigma_b\|}{n_b}}  \leq m\sigma^2_{\max}\sqrt{\frac{ t(1+t)}{n_b}} 
\end{equation*}
for all $(a,b,i)$
with probability at least $1 - 2kN \cdot p(t).$ Therefore,
\begin{equation*} 
J_2 \leq \sqrt{s (\bmu_a- \bmu_b)^{\top}\bSigma_a(\bmu_a - \bmu_b)} + m\sigma^2_{\max}\sqrt{\frac{ t(1+t)}{n_b}} 
\end{equation*}
holds with probability at least $1 - 2kN \cdot p(t) - kNe^{-s/2}$, for all $s,t \geq 0$.
\end{enumerate}
Using the estimation of $J_1$ and $J_2$, we can see that, for all $(a,b,i)$,
\begin{equation*}
(\bx_{a,i} - \bc_a)^{\top}(\bc_a - \bc_b) \leq  \sqrt{s(\bmu_a - \bmu_b)^{\top}\bSigma_a(\bmu_a- \bmu_b)} + 3m\sigma^2_{\max} \frac{1+t}{\sqrt{\min\{n_a, n_b\}} }
\end{equation*}
holds with probability at least $1 - kN(4 \cdot p(t) + e^{-s/2})$. Since $(\bu_{a,b})_i = \frac{1}{h_{a,b}} (\bx_{a,i} - \bc_a)^{\top}(\bc_a - \bc_b),$, if $N \geq \frac{16 \sigma_{\max}^2 m(1+t)}{\Delta^2 w_{\min}}$, then by \eqref{eq:gmm-h2} there hold,
\begin{align} \label{eq:gmm-tau}
\tau_{a,b} = \max \{ \max\{\bu_{a,b}\} , \max\{\bu_{b,a}\}  \}
\leq 2 \sqrt{s}\sigma_{\max} + \frac{ 6m\sigma^2_{\max}(1+t)}{\Delta \sqrt{Nw_{\min}}}.
\end{align}

\paragraph{Proximity condition for Gaussian mixture model}
By combing~\eqref{eq:gmm-upbd}, \eqref{eq:gmm-h1} and \eqref{eq:gmm-tau}, we have shown the proximity condition is satisfied with probability at least $1 - kN (5 \cdot p(t) + e^{-s/2})$ if
\begin{equation*}
\Delta  \geq \frac{2 \sigma_{\max}}{\sqrt{w_{\min}}} + 4\sigma_{\max} \sqrt{s} + 2 \sigma_{\max} (4\sqrt{k} + 1)\sqrt{\frac{m(1+t)}{Nw_{\min}}}  +  \frac{ 6m\sigma^2_{\max}(1+t)}{\Delta \sqrt{Nw_{\min}}},
\end{equation*}
provided that $N \geq \frac{16 \sigma_{\max}^2 m(1+t)}{\Delta^2 w_{\min}}$. These two inequalities are in turn implied by
\begin{equation} \label{eq:gmm-sep1}
\Delta  \geq \frac{2 \sigma_{\max}}{\sqrt{w_{\min}}} + 4\sigma_{\max} \sqrt{s} + 10 \sigma_{\max} \sqrt{\frac{km(1+t)}{Nw_{\min}}} +  \frac{ 6m\sigma_{\max}(1+t)}{ \sqrt{N}} 
\end{equation}
Here by choosing $t =\max \left\{ 8\log(kN^{1+\gamma}) /m, \sqrt{ 8\log(kN^{1+\gamma}) /m} \right\}$ and $s = 2\log(kN^{1 + \gamma})$ where $\gamma>0$, then the proximity condition holds 
with probability at least
\begin{equation*}
1 - kN (5 \cdot p(t) + e^{-s/2}) \geq 1 - 6N^{-\gamma}.
\end{equation*}
To simplify the expression, we assume $N=({m^2 k^2 \log(k)}/{w_{\min}}) u$, where $u \gg 1$. Denote $q(N;m,k,w_{\min})$ the sum of the last two terms of \eqref{eq:gmm-sep1} divided by $\sigma_{\max}$. We have the following asymptotic analysis:
\begin{equation*}
q(N;m,k,w_{\min}) \leq \sqrt{{\cal O} \left( \frac{1 + \log(km) + \log(u)}{kmu} \right) } + {\cal O} \left( \frac{1}{\sqrt{u}} + \frac{\log(k)}{k \sqrt{u}} + \frac{\log(N)}{\sqrt{N}} \right) = o(1).
\end{equation*}
This completes the proof of Corollary~\ref{cor:gmm}.

 \end{proof}

\section*{Acknowledgement}

Y.~Li, S.~Ling, T.~Strohmer, and K.~Wei acknowledge support from the NSF via grants DMS 1620455 and DMS 1737943.

\bibliographystyle{abbrv}
\bibliography{arXiv-kmeans-R1}

\begin{thebibliography}{10}

\bibitem{AchM05}
D.~Achlioptas and F.~McSherry.
\newblock On spectral learning of mixtures of distributions.
\newblock In {\em International Conference on Computational Learning Theory},
  pages 458--469. Springer, 2005.

\bibitem{AloiseDPP09}
D.~Aloise, A.~Deshpande, P.~Hansen, and P.~Popat.
\newblock {NP}-hardness of {E}uclidean sum-of-squares clustering.
\newblock {\em Machine Learning}, 75(2):245--248, 2009.

\bibitem{AminiL18}
A.~A. Amini and E.~Levina.
\newblock On semidefinite relaxations for the block model.
\newblock {\em The Annals of Statistics}, 46(1):149--179, 2018.

\bibitem{ArthurMR11}
D.~Arthur, B.~Manthey, and H.~R{\"o}glin.
\newblock Smoothed analysis of the k-means method.
\newblock {\em Journal of the ACM (JACM)}, 58(5):19, 2011.

\bibitem{AwasBCKVW15}
P.~Awasthi, A.~S. Bandeira, M.~Charikar, R.~Krishnaswamy, S.~Villar, and
  R.~Ward.
\newblock Relax, no need to round: Integrality of clustering formulations.
\newblock In {\em Proceedings of the 2015 Conference on Innovations in
  Theoretical Computer Science}, pages 191--200. ACM, 2015.

\bibitem{AwasthiS12}
P.~Awasthi and O.~Sheffet.
\newblock Improved spectral-norm bounds for clustering.
\newblock In {\em APPROX-RANDOM}, pages 37--49. Springer, 2012.

\bibitem{BenN01}
A.~Ben-Tal and A.~Nemirovski.
\newblock {\em Lectures on Modern Convex Optimization: Analysis, Algorithms,
  and Engineering Applications}.
\newblock SIAM, 2001.

\bibitem{BoydV04}
S.~Boyd and L.~Vandenberghe.
\newblock {\em Convex Optimization}.
\newblock Cambridge University Press, 2004.

\bibitem{Das99}
S.~Dasgupta.
\newblock Learning mixtures of gaussians.
\newblock In {\em Foundations of Computer Science, 1999. 40th Annual Symposium
  on}, pages 634--644. IEEE, 1999.

\bibitem{DuFG99}
Q.~Du, V.~Faber, and M.~Gunzburger.
\newblock Centroidal {V}oronoi tessellations: Applications and algorithms.
\newblock {\em SIAM Review}, 41(4):637--676, 1999.

\bibitem{GolubV96}
G.~H. Golub and C.~F. Van~Loan.
\newblock {\em Matrix Computations}.
\newblock The Johns Hopkins University Press, 3rd edition, 1996.

\bibitem{iguchi2015tightness}
T.~Iguchi, D.~G. Mixon, J.~Peterson, and S.~Villar.
\newblock On the tightness of an {SDP} relaxation of k-means.
\newblock {\em arXiv preprint arXiv:1505.04778}, 2015.

\bibitem{IguchiMPV17}
T.~Iguchi, D.~G. Mixon, J.~Peterson, and S.~Villar.
\newblock Probably certifiably correct k-means clustering.
\newblock {\em Mathematical Programming}, 165(2):605--642, 2017.

\bibitem{KanV09}
R.~Kannan, S.~Vempala, et~al.
\newblock Spectral algorithms.
\newblock {\em Foundations and Trends in Theoretical Computer Science},
  4(3--4):157--288, 2009.

\bibitem{KumarK10}
A.~Kumar and R.~Kannan.
\newblock Clustering with spectral norm and the k-means algorithm.
\newblock In {\em Foundations of Computer Science (FOCS), 2010 51st Annual IEEE
  Symposium on}, pages 299--308. IEEE, 2010.

\bibitem{LingS18}
S.~Ling and T.~Strohmer.
\newblock Certifying global optimality of graph cuts via semidefinite
  relaxation: A performance guarantee for spectral clustering.
\newblock {\em arXiv preprint arXiv:1806.11429}, 2018.

\bibitem{Lloyd82}
S.~Lloyd.
\newblock Least squares quantization in {PCM}.
\newblock {\em IEEE Transactions on Information Theory}, 28(2):129--137, 1982.

\bibitem{LuZ16}
Y.~Lu and H.~H. Zhou.
\newblock Statistical and computational guarantees of {L}loyd's algorithm and
  its variants.
\newblock {\em arXiv preprint arXiv:1612.02099}, 2016.

\bibitem{MahNV09}
M.~Mahajan, P.~Nimbhorkar, and K.~Varadarajan.
\newblock The planar k-means problem is {NP}-hard.
\newblock In {\em International Workshop on Algorithms and Computation}, pages
  274--285. Springer, 2009.

\bibitem{MixonVW17}
D.~G. Mixon, S.~Villar, and R.~Ward.
\newblock Clustering subgaussian mixtures by semidefinite programming.
\newblock {\em Information and Inference: A Journal of the IMA}, 6(4):389--415,
  2017.

\bibitem{PengW07}
J.~Peng and Y.~Wei.
\newblock Approximating k-means-type clustering via semidefinite programming.
\newblock {\em SIAM Journal on Optimization}, 18(1):186--205, 2007.

\bibitem{Selim84}
S.~Z. Selim and M.~A. Ismail.
\newblock k-means-type algorithms: A generalized convergence theorem and
  characterization of local optimality.
\newblock {\em IEEE Transactions on Pattern Analysis and Machine Intelligence},
  6(1):81--87, 1984.

\bibitem{Tropp12}
J.~A. Tropp.
\newblock User-friendly tail bounds for sums of random matrices.
\newblock {\em Foundations of Computational Mathematics}, 12(4):389--434, 2012.

\bibitem{Vattani11}
A.~Vattani.
\newblock k-means requires exponentially many iterations even in the plane.
\newblock {\em Discrete and Computational Geometry}, 45(4):596--616, 2011.

\bibitem{VemW04}
S.~Vempala and G.~Wang.
\newblock A spectral algorithm for learning mixture models.
\newblock {\em Journal of Computer and System Sciences}, 68(4):841--860, 2004.

\bibitem{Ver12}
R.~Vershynin.
\newblock Introduction to the non-asymptotic analysis of random matrices.
\newblock In Y.~C. Eldar and G.~Kutyniok, editors, {\em Compressed Sensing:
  Theory and Applications}, chapter~5. Cambridge University Press, 2012.

\bibitem{Wright1997}
S.~J. Wright.
\newblock {\em Primal-Dual Interior-Point Methods}.
\newblock SIAM, 1997.

\bibitem{yang2015sdpnal+}
L.~Yang, D.~Sun, and K.-C. Toh.
\newblock {SDPNAL+}: a majorized semismooth {Newton-CG} augmented {Lagrangian}
  method for semidefinite programming with nonnegative constraints.
\newblock {\em Mathematical Programming Computation}, 7(3):331--366, 2015.

\bibitem{zhao2010newton}
X.-Y. Zhao, D.~Sun, and K.-C. Toh.
\newblock A {Newton-CG} augmented {L}agrangian method for semidefinite
  programming.
\newblock {\em SIAM Journal on Optimization}, 20(4):1737--1765, 2010.

\end{thebibliography}

\end{document}